\documentclass[12pt,a4paper,leqno]{amsart}

\title[]{The Gabor wave front set}

\author[L. Rodino]{Luigi Rodino}

\address{Department of Mathematics, University of Turin, Via Carlo Alberto 10, 10123 Torino (TO), Italy.}

\email{luigi.rodino@unito.it}

\author[P. Wahlberg]{Patrik Wahlberg}

\address{Department of Mathematics, University of Turin, Via Carlo Alberto 10, 10123 Torino (TO), Italy.}

\email{patrik.wahlberg@unito.it}

\usepackage{latexsym}
\usepackage{amsmath}
\usepackage{amssymb}
\usepackage{amsthm}
\usepackage{amsfonts}
\usepackage{mathrsfs}
\usepackage{calc}
\usepackage{cite}
\usepackage{color}

\numberwithin{equation}{section}          

\newtheorem{thm}{Theorem}
\numberwithin{thm}{section}

\newcommand{\rubrik}{}
\newtheorem{prop}[thm]{Proposition}
\newtheorem{cor}[thm]{Corollary}

\theoremstyle{definition}

\newtheorem{defn}[thm]{Definition}
\newtheorem{example}[thm]{Example}

\theoremstyle{remark}

\newtheorem{rem}[thm]{Remark}              

\newcommand{\pd}[1] {\partial ^#1}
\newcommand{\pdd}[2] {\partial_{#1} ^#2}

\newcommand{\ro}{\mathbb R}
\newcommand{\no}{\mathbb N}
\newcommand{\rr}[1]{\mathbb R^{#1}}
\newcommand{\nn}[1]{\mathbb N^{#1}}

\newcommand{\zz}[1]{\mathbb Z^{#1}}

\newcommand{\ep}{\varepsilon}
\newcommand{\fy}{\varphi}
\newcommand{\lam}{\lambda}
\newcommand{\Lam}{\Lambda}

\newcommand{\supp}{\operatorname{supp}}

\newcommand{\wpr}{{\text{\footnotesize $\#$}}}

\newcommand{\eabs}[1]{\langle #1\rangle}

\newcommand{\Sp}{\operatorname{Sp}}
\newcommand{\GL}{\operatorname{GL}}
\newcommand{\charac}{\operatorname{char}}
\newcommand{\conesupp}{\operatorname{conesupp}}

\newcommand{\cS}{\mathscr{S}}
\newcommand{\cD}{\mathscr{D}}

\newcommand{\wt}{\widetilde}
\newcommand{\wh}{\widehat}

\def\la{\langle}
\def\ra{\rangle}

\begin{document}

\begin{abstract}
We define the Gabor wave front set $WF_G(u)$ of a tempered distribution $u$ in terms of rapid decay of its Gabor coefficients in a conic subset of the phase space.
We show the inclusion
\begin{equation*}
WF_G(a^w(x,D) u) \subseteq WF_G(u), \quad u \in \cS'(\rr d), \quad a \in S_{0,0}^0,
\end{equation*}
where $S_{0,0}^0$ denotes the H\"ormander symbol class of order zero and parameter values zero.
We compare our definition with other definitions in the literature, namely the classical and the global wave front sets of H\"ormander, and the $\cS$-wave front set of Coriasco and Maniccia.
In particular, we prove that the Gabor wave front set and the global wave front set of H\"ormander coincide.
\end{abstract}


\maketitle

\let\thefootnote\relax\footnotetext{Corresponding author: Patrik Wahlberg,
phone +390116702813, fax +390116702878. Authors' address: Dipartimento di Matematica, Universit{\`a} di Torino, Via Carlo Alberto 10, 10123 Torino (TO), Italy. Email \{luigi.rodino,patrik.wahlberg\}@unito.it.}

\section{Introduction}

Gabor frames are the appropriate tools for many problems in time-frequency analysis, with applications in signal processing and related issues in function space theory and numerical analysis, see for example \cite{Luef1,Daubechies1,Feichtinger1,FeiGro1,Grochenig1,Grochenig3,Janssen1,Strohmer1,Walnut1}.
Recent applications of Gabor frames concern also the analysis of partial differential equations and pseudo-differential equations, \cite{Benyi1,Benyi2,Benyi3,Cordero1,Cordero2,Cordero3,Grochenig2,Rochberg1,Sjostrand1,Sjostrand2,Wang1,Wang2}.
In the context of the modern theory of linear partial differential operators, a basic ingredient is the wave front set of a function, or distribution, as defined originally in \cite{Hormander2},
the \textit{classical wave front set} let us say, and modified subsequentially in different forms,
each adapted to the class of equations under investigation.
A natural question is how we may see wave front sets in terms of Gabor frames, or the short-term Fourier transform.
In this direction let us mention \cite{Johansson1,Pilipovic1,Pilipovic2}, recapturing the classical wave front set as the limit of sequences of wave front sets with respect to modulation spaces.

Here we present a different notion, even more natural in the context of time-frequency analysis, which we dub the \textit{Gabor wave front set}.
Denote translation by $T_x f(y)=f(y-x)$, modulation by $M_\xi f(y)=e^{i \la y,\xi \ra} f(y)$, $x,y,\xi \in \rr d$, and phase-space shifts by $\Pi(z)=M_\xi T_x$, $z=(x,\xi) \in \rr {2d}$.
Consider for real parameters $\alpha,\beta>0$ the lattice $\Lambda = \alpha \zz d \times \beta \zz d \subseteq \rr {2d}$. Fix a window function $\fy \neq 0$ in the Schwartz space $\cS(\rr d)$ and consider the Gabor system
$\{ \Pi(\lambda) \fy \}_{\lambda \in \Lambda}$, which is a frame if $\alpha$ and $\beta$ are assumed to be sufficiently small.
We then turn attention to the Gabor coefficients, defined for any tempered distribution $u \in \cS'(\rr d)$:
\begin{equation}\label{Gaborcoeff1}
c_\lambda = (u,\Pi(\lambda) \fy), \quad \lambda \in \Lambda.
\end{equation}

For $u \in \cS'(\rr d)$ we define the Gabor wave front set $WF_G(u)$ as a conic subset of $\rr {2d} \setminus \{(0,0) \}$ in the following way:
\begin{equation}\label{GaborWFdef}
\begin{aligned}
& \mbox{We say that $0 \neq z_0=(x_0,\xi_0) \notin WF_G(u)$ if there exists an open} \\
& \mbox{conic neighborhood $\Gamma \subseteq \rr {2d} \setminus \{(0,0) \}$ of $z_0$ such that} \\
& \qquad \qquad \qquad \qquad \sup_{\lambda \in \Lambda \cap \Gamma} \eabs{\lambda}^N |c_\lambda| < + \infty \quad \forall N \geq 0.
\end{aligned}
\end{equation}
Let us emphasize that cones are taken here with respect to the whole of the phase space variables $z=(x,\xi)$,
i.e. $z \in \Gamma \subseteq \rr {2d} \setminus \{(0,0) \}$ implies $tz \in \Gamma$ for any $t>0$.
So in \eqref{GaborWFdef} we assign a symmetrical role to translations and modulations,
by imposing a uniform rapid decay of the Gabor coefficients $c_\lambda$ for all the points of the lattice $\Lambda$ belonging to $\Gamma$.

It can be deduced from \eqref{GaborWFdef} the invariance under phase-space shifts:
\begin{equation}\label{phaseinvariance}
WF_G(\Pi(z) u) = WF_G(u) \quad \mbox{for every} \quad z \in \rr{2d}.
\end{equation}
More generally we shall prove that if a symbol of a pseudodifferential operator $a$ belongs to the class $S_{0,0}^0$, i.e. $\sup_{z \in \rr{2d}}| \pdd z \gamma a(z)|$ is finite for all multiindices $\gamma \in \nn {2d}$, then
\begin{equation}\label{microlocal1}
WF_G( a^w(x,D) u) \subseteq  WF_G(u), \quad u \in \cS'(\rr d),
\end{equation}
where $a^w(x,D)$ is the Weyl quantization of $a$.
The properties \eqref{phaseinvariance}, \eqref{microlocal1} are coherent with the time-frequency approach,
but they may certainly look striking, since, as is well known, the classical wave front set is not invariant under translations, and the inclusion \eqref{microlocal1} may fail for it when $a \in S_{0,0}^0$.

We observe that the Gabor wave front set exists already, under a completely different guise.
In fact, in \cite{Hormander1} H\"ormander introduced a \textit{global wave front set} for distributions $u \in \cS'(\rr d)$, addressed to the study of quadratic hyperbolic operators.
A similar notion was used in \cite{Nakamura1} in the study of propagation of micro-singularities for Schr\"odinger equations.

With respect to the well known classical wave front set (cf. \cite{Hormander2,Hormander0}), the definition of the global wave front set in \cite{Hormander1} has been, unfortunately, almost ignored in the literature, whereas it will represent indeed the starting point of our discussion.
So, in the next Section \ref{preliminaries}, devoted to preliminaries, we present as a service to the reader a largely self-contained account of certain parts of \cite{Hormander1}, rededucing some results.
In Section \ref{rapidcone}, as an intermediate step between the Gabor wave front set  and the global wave front set of \cite{Hormander1}, we shall prove that $WF_G(u)$ can be equivalently defined in a continuous setting by the short-time Fourier transform, i.e. by allowing $\lambda$ in \eqref{GaborWFdef} to belong to $\rr {2d}$.
The identity between the global wave front set and the Gabor wave front set will be finally proved in Section \ref{wfequality}.
Our arguments will show that the estimates \eqref{GaborWFdef} do not depend on the choice of the nonzero Schwartz window function $\fy$ as long as $\fy$ and $\Lambda$ generate a frame.
Combining with the results in \cite{Grochenig1,Grochenig2,Holst1,Sjostrand1,Sjostrand2}, we shall also prepare the necessary tools to prove in Section \ref{wfinclusion} the inclusion for $a \in S_{0,0}^0$ and $u \in \cS'(\rr d)$
\begin{equation}\label{microlocal2}
WF_G( a^w(x,D) u) \subseteq  WF_G(u) \cap \conesupp(a),
\end{equation}
strengthening \eqref{microlocal1}.
Such inclusions remained outside the results for the global wave front set in \cite{Hormander1}, since attention was limited there to the particular case of the Shubin symbols (cf. \cite{Shubin1}).
In some sense \eqref{microlocal2} shows now that also the ``bad'' symbol class $S_{0,0}^0$ can be micro-localized by means of time-frequency methods, if the appropriate definition of wave front set is chosen.

Section \ref{swf} is devoted to examples, in which we compare $WF_G(u)$ with the classical wave front set and the $\cS$-wave front set of \cite{Coriasco1}.
Though we do not give applications of \eqref{microlocal2} in the present paper, we hope, on the one hand, that the Gabor wave front set will be of some use in Signal Theory.
On the other hand, applications are possible in the study of propagation of micro-singularities for Schr\"odinger equations:
\begin{equation}\label{schrodinger1}
i \partial_t u = p^w(x,D) u, \quad u(x,0) = u_0(x).
\end{equation}
In particular when the real-valued Hamiltonian $p(z)$, $z=(x,\xi)$, in \eqref{schrodinger1} belongs to the class $S_{0,0}^2$, i.e. $\pdd z \gamma p \in S_{0,0}^0$ for $|\gamma|=2$, the Schr\"odinger propagator is a Fourier integral operator of the type considered in \cite{Asada1} (see \cite{Bony1,Tataru1}).
We may then expect to combine our definition of $WF_G(u)$ with the results of \cite{Cordero2,Cordero3,Cordero4}, analyzing such Fourier integral operators in the context of Gabor frames.

\section{Preliminaries}\label{preliminaries}

The transpose of a matrix $A$ is denoted $A^t$, and the inverse transpose of $A \in \GL(d,\ro)$ is written $A^{-t}$.
An open ball in $\rr d$ of radius $\delta>0$ centered at the origin is denoted $B_d^\delta = \{ x \in \rr d:  \ |x| < \delta \}$,
and the unit sphere in $\rr d$ is denoted $S_{d-1}=\{ x \in \rr d: \, |x|=1 \}$.
The overline $\overline A$ denotes either the closure of a measurable set $A \subseteq \rr d$,
or as in $\overline f$ the complex conjugate of a function $f$ on $\rr d$, the choice being clear from the context.
The Fourier transform of $f \in \mathscr S(\rr d)$ (the Schwartz space) is defined by
$$
\mathscr{F} f(\xi) = \wh f(\xi) = \int_{\rr d} f(x) e^{- i \la x, \xi \ra} dx,
$$
where $\la x, \xi \ra$ denotes the inner product on $\rr d$.
The standard multiindex notation for partial differential operators is used, and $D=(D_1,\dots,D_d)$ where $D_j=-i \partial/\partial x_j$.
As in the Introduction we denote translation by $T_x f(y)=f(y-x)$, modulation by $M_\xi f(y)=e^{i \la y, \xi \ra} f(y)$, $x,y,\xi \in \rr d$, and
phase space translation by $\Pi(z) = M_\xi T_x$, $z=(x,\xi) \in \rr {2d}$.

The cross-Wigner distribution of $f,g \in \mathscr S(\rr d)$ is defined by
\begin{equation}\nonumber
W(f,g)(x,\xi) = \int_{\rr d} f(x+\tau/2) \ \overline{g(x-\tau/2)} \ e^{- i \la \tau, \xi \ra} \ d \tau, \quad (x,\xi) \in \rr {2d}.
\end{equation}
We abbreviate $W(f,f)=W(f)$ and denote $\eabs{x} = (1+|x|^2)^{1/2}$ for $x \in \rr d$.
We write $f (x) \lesssim g (x)$ provided there exists $C>0$ such that $f (x) \leq C g(x)$ for all $x$ in the domain of $f$ and $g$.
If $f \lesssim g$ and $g \lesssim f$ then we use $f \asymp g$.
For $s \in \ro$, the weighted Lebesgue space $L_s^1(\rr d)$ is defined by the norm $\| f \|_{L_s^1(\rr d)} = \| f \eabs{\cdot}^s \|_{L^1(\rr d)}$.
A set $A \subseteq \rr d$ is called conic if it is invariant with respect to multiplication by all positive reals.

Given a window function $\varphi \in \mathscr S(\rr d) \setminus \{ 0 \}$, the short-time Fourier transform (STFT) is defined by
\begin{equation}\nonumber
V_\varphi f(z) = ( f, \Pi(z) \varphi ), \quad f \in \mathscr S'(\rr d), \quad z \in \rr {2d},
\end{equation}
where $(\cdot,\cdot)$ denotes the conjugate linear action of $\mathscr S'$ on $\mathscr S$,
consistent with the inner product $(\cdot,\cdot)_{L^2}$ which is conjugate linear in the second argument.
The function $z \mapsto V_\varphi f(z)$ is smooth and bounded by $C \eabs{z}^k$ for some $C,k \geq 0$.
If $\varphi \in \cS(\rr d)$, $\| \varphi \|_{L^2}=1$ and $f \in \cS'(\rr d)$,
the STFT inversion formula reads (cf. \cite[Corollary 11.2.7]{Grochenig1})
\begin{equation}\label{STFTrecon}
(f,g ) = (2 \pi)^{-d} \int_{\rr {2d}} V_\varphi f(z) ( \Pi(z) \varphi,g ) \, dz, \quad g \in \cS(\rr d).
\end{equation}
This means that $V_\fy^* V_\fy = I$ on $\cS'(\rr d)$ where the adjoint $V_\fy^*$ is defined by
\begin{equation*}
(V_\fy^* F,g) = (2 \pi)^{-d} (F ,V_\fy g),
\quad g \in \cS(\rr d), \quad F \in \cS'(\rr {2d}).
\end{equation*}
We need a discrete version of the STFT inversion formula based on the notion of Gabor frame \cite{Grochenig1}, outlined as follows.
As in the Introduction, define for real parameters $\alpha,\beta >0$ such that $\alpha \beta \leq 2 \pi$ the lattice $\Lam = \alpha \zz d \times \beta \zz d \subseteq \rr {2d}$.
For a window function $\fy \in L^2(\rr d) \setminus \{ 0 \}$ the collection $\{ \Pi(\lam) \varphi \}_{\lam \in \Lam}$ is called a Gabor frame for $L^2(\rr d)$ provided there exists constants $A,B>0$ such that
\begin{equation}\label{frameL2}
A \| f \|_{L^2}^2 \leq \sum_{\lam \in \Lam} | (f,\Pi(\lam) \fy)|^2 \leq B \| f \|_{L^2}^2, \quad f \in L^2(\rr d).
\end{equation}
The Gabor frame operator
\begin{equation*}
S f = \sum_{\lam \in \Lam} (f,\Pi(\lam) \fy ) \, \Pi(\lam) \fy
\end{equation*}
is then bounded, positive and invertible on $L^2(\rr d)$.
Using the \emph{canonical dual} window
\begin{equation*}
\wt \fy = S^{-1} \fy,
\end{equation*}
one can reconstruct $f$ from its Gabor coefficients $\{ (f,\Pi(\lam) \fy) \}_{\lam \in \Lam}$ as
\begin{equation}\label{Gaborrecon}
f = \sum_{\lam \in \Lam} (f,\Pi(\lam) \fy) \, \Pi(\lam)  \wt \fy, \quad f \in L^2(\rr d),
\end{equation}
where the sum converges unconditionally in $L^2$.

Let $\fy \in \cS(\rr d)\setminus \{ 0 \}$ and let  $\alpha, \beta>0$ be sufficiently small so that $\{ \Pi(\lam) \varphi \}_{\lam \in \Lam}$ is a Gabor frame for $L^2(\rr d)$.
Then $\wt \fy = S^{-1} \fy \in \cS(\rr d)$ as proved by Janssen \cite{Janssen1}.

By results of Feichtinger, Gr\"ochenig and Leinert \cite{FeiGro1,Grochenig3},
Gabor frame theory extends to weighted modulation spaces, introduced by Feichtinger \cite{Feichtinger1}, as follows.
Let $w \in L_{\rm loc}^\infty(\rr {2d})$ be a positive weight function, moderate in the sense of
\begin{equation*}
w(z+\zeta) \lesssim w(z) \eabs{\zeta}^n, \quad z,\zeta \in \rr {2d},
\end{equation*}
for some $n \geq 0$. If $p,q \in [1,+\infty]$ then the modulation space norm of $f \in \cS'(\rr d)$ is defined by
\begin{equation*}
\| f \|_{M_w^{p,q}} = \left( \int_{\rr d} \left( \int_{\rr d} \left| V_\varphi f(x,\xi) w(x,\xi) \right|^p \, dx \right)^{q/p} \, d \xi \right)^{1/q}
\end{equation*}
if $p, q<\infty$, and modified as usual if $p=\infty$ or $q=\infty$.
The modulation space $M_w^{p,q}(\rr d)$ consists of all $f \in \cS'(\rr d)$ such that $\| f \|_{M_w^{p,q}}<\infty$.   It is a Banach space.
Different windows in $\cS(\rr d)\setminus \{ 0 \}$ yield equivalent norms and we have the embeddings
\begin{align*}
& \cS(\rr d) \subseteq M_{w_1}^{p_1,q_1}(\rr d)  \subseteq M_{w_2}^{p_2,q_2}(\rr d) \subseteq \cS'(\rr d), \\
& \quad p_1 \leq p_2, \quad q_1 \leq q_2, \quad w_2 \lesssim w_1.
\end{align*}
We abbreviate $M_w^{p,p}=M_w^p$.
The Gabor frame reconstruction formula \eqref{Gaborrecon} extends to $f \in M_w^{p,q}(\rr d)$, for all weighted modulation spaces $M_w^{p,q}(\rr d)$,
$p,q \in [1,+\infty]$, with unconditional convergence in the norm $\| \cdot \|_{M_w^{p,q}}$ for $p,q<\infty$, and the weak unconditional convergence
\begin{equation}\label{Gaborreconweak}
(f,g) = \sum_{\lam \in \Lam} (f,\Pi(\lam) \fy) \, (\Pi(\lam)  \wt \fy,g), \quad g \in \cS(\rr d),
\end{equation}
if $p=\infty$ or $q=\infty$. Moreover we have the norm equivalence
\begin{equation*}
\| f \|_{M_w^{p,q}}  \asymp \left( \sum_{n \in \zz d} \left( \sum_{k \in \zz d} \left| V_\varphi f(\alpha k, \beta n) w(\alpha k,\beta n) \right|^p \, \right)^{q/p} \,  \right)^{1/q},
\quad f \in M_w^{p,q},
\end{equation*}
for $p,q < \infty$, and conventionally modified otherwise.

The final fact concerning modulation spaces that we need is the expression of $\cS'$ as a union of weighted modulation spaces.
If $v_s(z) = \eabs{z}^s$ for $s \in \ro$ and $z \in \rr {2d}$ then
\begin{equation}\label{sprimemodulation}
\cS'(\rr d) = \bigcup_{s \leq 0} M_{v_s}^\infty(\rr d).
\end{equation}
For a richer background on Gabor frames and modulation spaces we refer to Gr\"ochenig's book \cite{Grochenig1}.

Next we define the notion of conical support of a function, or distribution, on $\rr n$.

\begin{defn}
For $a \in \cD'(\rr n)$ the conic support $\conesupp (a)$ is the set of all
$x \in \rr n \setminus \{ 0 \}$ such that any conic open set $\Gamma_x \subseteq \rr n \setminus \{ 0 \}$ containing $x$ satisfies:
$$
\overline{\supp (a) \cap \Gamma_x} \quad \mbox{is not compact in} \quad \rr n.
$$
\end{defn}

It follows that $\conesupp (a)$ is a closed subset of $\rr n \setminus \{ 0 \}$.
In the sequel $n=2d$.

For pseudodifferential operators we work with symbols in either the Shubin classes $G^m$ \cite{Shubin1}, or the H\"ormander classes $S_{\rho,\delta}^m$ with the restriction $\rho=\delta=0$ \cite{Folland1,Hormander0}.
First we discuss the Shubin classes.

\begin{defn}\label{shubinclasses1}
For $m\in \ro$, $G^m$ is the subspace of all
$a \in C^\infty(\rr {2d})$ such that for every
$\alpha,\beta \in \nn d$ there exists a constant $C_{\alpha,\beta}>0$ so that
\begin{equation}\label{symbolestimate1}
|\partial_x^\alpha \partial_\xi^\beta a(x,\xi)| \le C_{\alpha,\beta} \langle (x,\xi) \rangle^{m-|\alpha|-|\beta|}
\end{equation}
is satisfied for every $(x,\xi)\in \rr {2d}$. $G^m$ is a Fr\'echet space with respect to the seminorms defined by
\begin{equation*}
\sup_{(x,\xi) \in \rr {2d}} \la (x,\xi) \ra^{-m+|\alpha|+|\beta|} \left| \partial_x^\alpha \partial_\xi^\beta a(x,\xi) \right|, \quad (\alpha,\beta) \in \nn {2d}.
\end{equation*}
\end{defn}
We have $\bigcap_{m \in \ro} G^m = \mathscr S(\rr {2d})$ and we denote $G^\infty = \bigcup_{m \in \ro} G^m$.
Let $(a_j)_{j \geq 0}$ be a sequence of symbols such that $a_j \in G^{m_j}$ and $m_j \rightarrow - \infty$ as $j \rightarrow +\infty$, and set $m=\max_{j \geq 0} m_j$.
Then there exists a symbol $a \in G^{m}$, unique modulo $\cS(\rr {2d})$, such that
$$
a - \sum_{j=0}^{n-1} a_j \in G^{m_n'}, \quad n \geq 1, \quad m_n' = \max_{j \geq n} \, m_j
$$
(cf. \cite[Proposition 23.1]{Shubin1}).
This is called an asymptotic expansion and denoted $a \sim \sum_{j \geq 0} a_j$.

The Weyl quantization is the map from symbols to operators defined by
\begin{equation}\nonumber
a^w(x,D) f(x) = (2 \pi)^{-d} \iint_{\rr {2d}} e^{i \la x-y,\xi \ra} a \left( \frac{x+y}{2},\xi \right)  f(y) \ dy \ d \xi
\end{equation}
for $a \in \mathscr S(\rr {2d})$ and $f \in \mathscr S(\rr d)$. The latter conditions can be relaxed in various ways.
In particular, if $a \in G^m$ for $m \in \ro$ then $a^w(x,D)$ is continuous on $\mathscr S(\rr d)$, and extending by duality it follows that $a^w(x,D)$ is continuous on $\mathscr S'(\rr d)$ \cite{Shubin1}.
By the Schwartz kernel theorem, any continuous linear operator $\cS(\rr d) \mapsto \cS'(\rr d)$
can be written as a Weyl quantization for a unique symbol $a \in \cS'(\rr {2d})$.
The Weyl quantization can be expressed in terms of the cross-Wigner distribution as
\begin{equation}\label{weylwigner1}
(a^w(x,D) f,g) = (2 \pi)^{-d} (a,W(g,f) ), \quad f,g \in \mathscr S(\rr d), \quad a \in \mathscr S'(\rr {2d}).
\end{equation}
The Weyl product $\wpr$ is the product on symbol pairs corresponding to operator composition:
$$
a^w(x,D) \, b^w(x,D) = (a \wpr b)^w(x,D).
$$
The Weyl product is a bilinear continuous map
\begin{equation}\label{weylproduct}
\wpr: \, G^m \times G^n \mapsto G^{m+n}.
\end{equation}
We have the following asymptotic expansion for the Weyl product of $a \in G^m$ and $b \in G^n$, $m,n \in \ro$ (cf. \cite[Theorem~23.6 and Problem~23.2]{Shubin1}):
\begin{equation}\label{calculuscomposition1}
a \wpr b(x,\xi) \sim \sum_{\alpha, \beta \geq 0} \frac{(-1)^{|\beta|}}{\alpha! \beta!} \ 2^{-|\alpha+\beta|}
D_x^\beta \pdd \xi \alpha a(x,\xi) \, D_x^\alpha \pdd \xi \beta b(x,\xi).
\end{equation}

We also need to consider the Kohn--Nirenberg quantization, defined by
\begin{equation}\nonumber
a(x,D) f(x) = (2 \pi)^{-d} \int_{\rr d}  e^{i \la x, \xi \ra} a ( x,\xi ) \wh f(\xi) \, d \xi, \quad a \in G^m, \quad f \in \mathscr S(\rr d).
\end{equation}
Since any continuous linear operator $\cS(\rr d) \mapsto \cS'(\rr d)$ has a unique Weyl symbol
in $\cS'(\rr {2d})$ as well as a unique Kohn--Nirenberg symbol in $\cS'(\rr {2d})$,
a bijective mapping from the Kohn--Nirenberg symbol to the Weyl symbol, denoted $T$,
can be defined for such operators.
This means that
\begin{equation}\label{KohnNirenberg2Weyl}
a(x,D)=(Ta)^w(x,D), \quad a \in \cS'(\rr {2d}),
\end{equation}
and $a \in G^m$ if and only if $Ta \in G^m$ (cf. \cite[Corollary 23.2]{Shubin1}).
Furthermore we have the asymptotic expansions for $a \in G^m$ (cf. \cite[Theorem 23.3]{Shubin1})
\begin{align}
Ta (x,\xi) & \sim \sum_{\alpha \geq 0} \frac1{\alpha!} \, \left( -\frac1{2} \right)^{|\alpha|}
D_x^\alpha \pdd \xi \alpha a(x,\xi), \label{calculusquant}  \\
T^{-1} a (x,\xi) & \sim \sum_{\alpha \geq 0} \frac1{\alpha!} \, \left(\frac1{2} \right)^{|\alpha|}
D_x^\alpha \pdd \xi \alpha a(x,\xi). \label{calculusquantb}
\end{align}

Next we introduce the H\"ormander symbol classes.

\begin{defn}\label{hormanderclasses}
For $m\in \ro$, $0 \leq \rho \leq 1$, $0 \leq \delta < 1$, $S_{\rho,\delta}^m$ is the subspace of all
$a \in C^\infty(\rr {2d})$ such that for every
$\alpha,\beta \in \nn d$ there exists a constant $C_{\alpha,\beta}>0$ so that
\begin{equation}\label{symbolestimate2}
|\partial_x^\alpha \partial_\xi^\beta a(x,\xi)| \le C_{\alpha,\beta} \eabs{\xi}^{m - \rho|\beta| + \delta |\alpha|}
\end{equation}
is satisfied for every $(x,\xi)\in \rr {2d}$. $S_{\rho,\delta}^m$ is a Fr\'echet space with respect to the seminorms defined by
\begin{equation*}
\sup_{(x,\xi) \in \rr {2d}} \eabs{\xi}^{-m + \rho|\beta| - \delta |\alpha|} \left| \partial_x^\alpha \partial_\xi^\beta a(x,\xi) \right|, \quad (\alpha,\beta) \in \nn {2d}.
\end{equation*}
\end{defn}

We will restrict to the special case $\rho=\delta=0$, i.e. $S_{0,0}^m$.
If $a \in S_{0,0}^m$ then $a^w(x,D)$ is continuous on $\cS(\rr d)$ \cite[Theorem~2.21]{Folland1},
extends to be continuous on $\cS'(\rr d)$,
and $T$ defined by \eqref{KohnNirenberg2Weyl} maps $S_{0,0}^m$ into itself continuously \cite[Theorem~2.37]{Folland1}.
The Weyl product is continuous
\begin{equation}\label{weylproduct2}
\wpr: \, S_{0,0}^m \times S_{0,0}^n \mapsto S_{0,0}^{m+n}
\end{equation}
(see e.g. \cite[Theorem~2.47]{Folland1}), but due to the lack of improved decay in $\xi$ when differentiating with respect to $\xi$, neither the asymptotic expansion \eqref{calculuscomposition1}
nor \eqref{calculusquant}, \eqref{calculusquantb} hold for $S_{0,0}^m$.
They do hold for $S_{\rho,\delta}^m$ as soon as $\rho>\delta$ (see \cite[Theorems~2.41~and~2.49]{Folland1}).
For this lack of powerful calculus rules, $S_{\rho,\rho}^m$ with $0 \leq \rho \leq 1$ are considered to be difficult symbol classes for pseudodifferential operators, whose analysis requires techniques quite different from the calculus available for $S_{\rho,\delta}^m$ with $0 \leq \delta < \rho \leq 1$.

H\"ormander \cite{Hormander1} introduced the following concepts in order to define a global type of wave front set as a conic subset of the phase space.

\begin{defn}\label{noncharacteristic2}
Given $a \in G^m$, a point in the phase space $z_0 \in T^*(\rr d) \setminus \{ (0,0) \}$ is called non-characteristic for $a$ provided there exists $A,\ep>0$ and an open conic set $\Gamma \subseteq T^*(\rr d) \setminus \{ (0,0) \}$ such that $z_0 \in \Gamma$ and
\begin{equation}\label{noncharlowerbound2}
|a(z )| \geq \ep \eabs{z}^m, \quad z \in \Gamma, \quad |z| \geq A.
\end{equation}
\end{defn}

\begin{rem}
As in \cite{Hormander1}, the condition \eqref{noncharlowerbound2} could be replaced by the apparently less restrictive condition
\begin{equation}\label{noncharlowerbound1}
|a(t z_0 )| \geq \ep_1 t^m, \quad t \geq A_1
\end{equation}
for some $A_1,\ep_1>0$, which actually implies \eqref{noncharlowerbound2} for an open conic set
$\Gamma \subseteq T^*(\rr d) \setminus \{ (0,0) \}$ containing $z_0$, and $A,\ep>0$.
In fact, assuming \eqref{noncharlowerbound1} for $z_0 \neq 0$, set for $0<r<|z_0|/2$
\begin{equation*}
\Gamma = \bigcup_{t>0} \left\{ z \in \rr {2d} \setminus \{0\}: \, \left| z - tz_0 \right| < r t \right\}.
\end{equation*}
By the mean value theorem and the estimates \eqref{symbolestimate1} there exists $C>0$ that depends on $m,z_0,A_1$ but not on $r$ such that, for $t \geq A_1$ and $| z - tz_0| < r t$,
\begin{align*}
| a(z) - a(t z_0) | \leq \sqrt{2} \, |z- t z_0 | \, \sup_{|w-t z_0| \leq r t} |\nabla a (w)|
\leq t^m C r,
\end{align*}
and thus we have, using \eqref{noncharlowerbound1}, for $t \geq A_1$ and $| z - tz_0| < r t$,
\begin{align*}
|a(z)| \geq |a(t z_0)|- | a(t z_0) -  a(z) | \geq \ep_1 t^m(1-C r {\ep_1}^{-1}) \geq \ep \eabs{z}^m,
\end{align*}
for $|z| \geq A$, where $A,\ep>0$, provided $r$ is chosen sufficiently small.
\end{rem}

For $a \in G^m$ we denote the characteristic set of $a$ by $\charac (a)$ and define it as the set of all $(x,\xi) \in T^*(\rr d) \setminus \{ (0,0) \}$ such that $(x,\xi)$ is not non-characteristic according to Definition \ref{noncharacteristic2}.
Note that
\begin{equation*}
\conesupp (a) \cup \charac (a) = T^*(\rr d) \setminus \{ (0,0) \}, \quad a \in G^m.
\end{equation*}

The definition of the global wave front set, which we shall simply denote by $WF(u)$ in the sequel, is as follows.

\begin{defn}\label{wavefront1}
\cite{Hormander1}
If $u \in \mathscr S'(\rr d)$ then $WF(u)$ is the set of all phase space points $(x,\xi) \in T^*(\rr d) \setminus \{ (0,0) \}$ such that $a \in G^m$ for some $m \in \ro$ and $a^w(x,D) u \in \mathscr S$ implies that $(x,\xi) \in \charac(a)$.
\end{defn}

For the benefit of non-expert readers, we recall from \cite{Hormander1} the basic properties of $WF(u)$, rededucing some results.
The set $WF(u)$ is closed and conic in $T^*(\rr d) \setminus \{ (0,0) \}$.
We have $(0,0) \neq (x,\xi) \notin WF(u)$ if and only if there exists $m \in \ro$ and $a \in G^m$ such that $a^w(x,D) u \in \mathscr S$ and $(x,\xi) \notin \charac (a)$.
According to the following result such a symbol $a \in G^m$ can be assumed
to satisfy $a \in G^0$, $0 \leq a \leq 1$, and $a(z)=1$ for $z \in \Gamma$ and $|z| \geq 1$, where
$\Gamma \subseteq T^*(\rr d) \setminus \{ (0,0) \}$ is open, conic and $z_0 \in \Gamma$.

\begin{prop}\label{noncharacteristic1}
If $0 \neq z_0 \notin WF(u)$ then there exists $b \in G^0$ and an open conic set
$\Gamma'' \subseteq T^*(\rr d) \setminus \{ (0,0) \}$ containing $z_0$
such that $0 \leq b \leq 1$, $b(z)=1$ for $z \in \Gamma''$ and $|z| \geq 1$,
and $b^w(x,D) u \in \mathscr S$.
\end{prop}

\begin{proof}
By assumption there exists $a \in G^m$,
an open conic set $\Gamma \subseteq T^*(\rr d) \setminus \{ (0,0) \}$ containing $z_0$, and $\ep,A>0$, such that
\eqref{noncharlowerbound2} holds and $a^w(x,D) u \in \mathscr S$.
The asymptotic expansion \eqref{calculuscomposition1} gives
$\overline{a} \wpr a = |a|^2 + a'$
where $a' \in G^{2m-2}$, and
the assumption $a^w(x,D) u \in \mathscr S$ gives $\overline{a}^w (x,D) a^w(x,D) u \in \mathscr S$.
We have for $z \in \Gamma$ and $|z| \geq A$, for some $C>0$,
\begin{equation*}
\begin{aligned}
|\overline{a} \wpr a(z)| & \geq |a(z)|^2 - |a'(z)|
\geq \ep^2 \eabs{z}^{2m} - C \eabs{z}^{2m-2} \\
& \geq \frac{\ep^2}{2} \eabs{z}^{2m}, \quad |z| \geq A_1,
\end{aligned}
\end{equation*}
provided $A_1>0$ sufficiently large.
It follows from this argument that we may assume $G^m \ni a=a_0+a_1$ where $a_1 \in G^{m-1}$, $a_0 \in G^m$, $a_0 \geq 0$, and
$$
a_0(z) \geq \ep \eabs{z}^{m}, \quad z \in \Gamma, \quad |z| \geq A.
$$

Let $\Gamma',\Gamma'' \subseteq T^*(\rr d) \setminus \{ (0,0) \}$ be open conic sets
such that $z_0 \in \Gamma''$,
$\overline{\Gamma'' \cap S_{2d-1}} \subseteq \Gamma'$, and
$\overline{\Gamma' \cap S_{2d-1}} \subseteq \Gamma$.
Let
$$
\chi(z) = \psi(z) \fy(z/|z|)
$$
where $\fy \in C_c^\infty(\rr {2d})$, $0 \leq \fy \leq 1$, $\supp(\fy) \subseteq \Gamma$, $\fy(z)=1$ for $z \in \Gamma' \cap S_{2d-1}$,
and where $\psi \in C^\infty(\rr {2d})$ is a cutoff function with $0 \leq \psi \leq 1$, $\psi(z)=0$ for $|z| \leq 1/2$ and $\psi(z)=1$ for $|z| \geq 1$.
This standard construction gives $\chi \in G^0$ such that $0 \leq \chi \leq 1$,
$\supp(\chi) \subseteq \Gamma$, $\conesupp(\chi) \subseteq \Gamma$, $\chi(z)=1$ when $z \in \Gamma'$ and $|z| \geq 1$.
Likewise we may take $b \in G^0$ such that $0 \leq b \leq 1$, $\supp(b) \subseteq \Gamma'$, $b(z)=1$ when $z \in \Gamma''$ and $|z| \geq 1$.
Set
$$
b_0(z) = \chi(z) a (z) + (1-\chi(z)) \, \ep \eabs{z}^m \in G^m.
$$
For $z \notin \Gamma$ we have $|b_0(z)| = \ep \eabs{z}^m$, whereas for $z \in \Gamma$ and $|z| \geq A$ we have,
for some $C>0$,
\begin{equation}\nonumber
\begin{aligned}
| b_0(z)| & = | \chi(z) a_0(z) + (1-\chi(z)) \, \ep \eabs{z}^m + \chi(z) \, a_1(z) | \\
& \geq \chi(z) a_0(z) + (1-\chi(z)) \, \ep \eabs{z}^m - \chi(z) \, | a_1(z) | \\
& \geq \ep \eabs{z}^m - C \eabs{z}^{m-1} \\
& \geq \frac{\ep}{2} \eabs{z}^m,
\end{aligned}
\end{equation}
in the last step after possibly augmenting $A>0$. This implies that $b_0 \in HG^{m,m}$ which denotes the set of hypoelliptic symbols in $G^m$ (cf.  \cite[Definition~25.1]{Shubin1}).
According to \cite[Theorem~25.1]{Shubin1} there exists $c \in HG^{-m,-m}$ such that $c \wpr b_0 = 1 + r$
where $r \in \mathscr S(\rr {2d})$.
Therefore
\begin{equation}\nonumber
b = b \wpr c \wpr b_0 - b \wpr r
= b \wpr c \wpr a  + b \wpr c \wpr (b_0-a) - b \wpr r.
\end{equation}
Since
$$
b_0-a = (1-\chi) (\ep \eabs{\cdot}^m - a),
$$
it follows that
$$
\supp(b) \cap \supp(b_0-a) \quad \mbox{is compact}.
$$
Hence by \eqref{calculuscomposition1} we have $b \wpr c \wpr (b_0-a) \in \cS(\rr {2d})$.
Taking into account $b \wpr r \in \cS(\rr {2d})$ and $a^w(x,D) u \in \cS$, we obtain finally
\begin{equation}\nonumber
\begin{aligned}
b^w(x,D) u
& = b^w(x,D) c^w(x,D) a^w(x,D) u  \\
& \qquad + (b \wpr c \wpr (b_0-a))^w(x,D) u - (b \wpr r)^w(x,D) u
\in \cS.
\end{aligned}
\end{equation}
\end{proof}

\begin{rem}\label{KohnNirenbergequivdef}
Note that the Kohn--Nirenberg quantization gives an equivalent definition in Definition \ref{wavefront1}.
In fact, if there exists $a \in G^m$ such that $a^w(x,D) u \in \mathscr S$ and $0 \neq z_0 \notin \charac (a)$,
then there exists an open conic set $\Gamma \subseteq T^*(\rr d) \setminus \{ (0,0) \}$ containing $z_0$
such that \eqref{noncharlowerbound2} holds for some $\ep,A>0$.
From \eqref{calculusquantb} we obtain $T^{-1} a = a + b$ where $b \in G^{m-1}$, and therefore for
$z \in \Gamma$ and $|z| \geq A$, we have fore some $C>0$
\begin{align*}
\left| (T^{-1} a)(z) \right| & \geq |a(z)| - |b(z)| \geq \ep \eabs{z}^m - C \eabs{z}^{m-1}
\geq \frac{\ep}{2} \eabs{z}^m
\end{align*}
provided $|z| \geq A_1$ for $A_1>0$ sufficiently large.
Thus $z_0 \notin \charac (T^{-1} a)$ and $(T^{-1} a)(x,D) u = a^w(x,D) u \in \cS$.
\end{rem}

The symplectic group $\Sp(d,\ro)$ consists of all matrices $A \in \GL( 2 d, \ro)$ that satisfy
$$
\sigma(Az,Az') = \sigma(z,z'), \quad z,z' \in \rr {2d},
$$
where the symplectic form $\sigma: \rr {2d} \times \rr {2d} \mapsto \ro$ is defined by
$$
\sigma((x,\xi), (x',\xi')) = \la x' , \xi \ra - \la x, \xi' \ra.
$$

To each symplectic matrix $\chi \in \Sp(d,\ro)$ is associated a unitary operator $U_\chi$ on $L^2(\rr d)$, determined up to a complex factor of modulus one, such that
\begin{equation}\label{symplecticoperator}
U_\chi^{-1} a^w(x,D) \, U_\chi = (a \circ \chi)^w(x,D), \quad a \in \cS'(\rr {2d})
\end{equation}
(cf. \cite{Folland1,Hormander0}).
$U_\chi$ is an homeomorphism on $\mathscr S$ and on $\mathscr S'$.
According to \cite[Proposition~2.2]{Hormander1} the global wave front set is symplectically invariant as follows
for $u \in \cS'(\rr d)$.
\begin{equation}\label{symplecticinvariance2}
(x,\xi) \in WF(u) \quad \Longleftrightarrow \quad \chi(x,\xi) \in WF( U_\chi u) \quad \forall \chi \in \Sp(d, \ro),
\end{equation}
or, in short, $WF( U_\chi u) = \chi WF(u)$ for all $\chi \in \Sp(d, \ro)$ and all $u \in \cS'(\rr d)$.

Finally we shall recall and prove some inclusions for the action of pseudodifferential operators with symbols in the Shubin classes $G^\infty$ on the global wave front set in terms of the conic support and the characteristic set.

\begin{prop}\label{wavefrontinclusion1}
If $u \in \mathscr S'(\rr d)$ and $a \in G^m$ then
\begin{align}
WF( a^w(x,D) u) \quad & \subseteq \quad WF(u) \bigcap \conesupp (a) \nonumber \\
& \subseteq \quad WF(u) \quad \subseteq \quad WF( a^w(x,D) u) \ \bigcup \ \charac (a). \label{wfinclusion2}
\end{align}
\end{prop}

\begin{proof}
The first inclusion in \eqref{wfinclusion2} is identical to \cite[Proposition~2.5]{Hormander1},
and the second inclusion is trivial, so it suffices to prove the third inclusion.
Suppose that $0 \neq z_0 \notin WF( a^w(x,D) u)$ and $z_0 \notin \charac (a)$.
By Definition \ref{noncharacteristic2} and Proposition \ref{noncharacteristic1} there exists $b \in G^0$ and an open conic set $\Gamma \subseteq T^* (\rr d) \setminus \{ (0,0) \}$ containing $z_0$ such that $b^w(x,D) \ a^w(x,D) u \in \mathscr S$ and
\begin{equation}\nonumber
\begin{aligned}
| a(z) | & \geq \ep \eabs{z}^m, \quad z \in \Gamma, \quad |z| \geq A, \\
b(z) & = 1, \quad z \in \Gamma, \quad |z| \geq A,
\end{aligned}
\end{equation}
for some $\ep,A>0$.
The asymptotic expansion \eqref{calculuscomposition1} gives
$b \wpr a = ba + c$
where $c \in G^{m-2}$. Therefore we have, provided $z \in \Gamma$ and $|z| \geq A$, for some $C>0$
\begin{equation}\nonumber
| b \wpr a(z) | \geq | b(z) a(z) | - |c(z)| \geq \ep \eabs{z}^{m} \left( 1 - \ep^{-1} C \eabs{z}^{-2} \right).
\end{equation}
By possibly augmenting $A$ it follows that $z_0 \notin \charac (b \wpr a)$. Thus $z_0 \notin WF(u)$, which means that we have proved the third inclusion in \eqref{wfinclusion2}.
\end{proof}

\begin{cor}
If $a \in G^m$, $u \in \cS'(\rr d)$ and $\charac (a) = \emptyset$, then
$$
WF( a^w(x,D) u) = WF(u).
$$
\end{cor}

The following result is \cite[Proposition~2.4]{Hormander1}.

\begin{prop}\label{conesuppwavefront}
If $u \in \cS'(\rr d)$, $a \in G^m$, and
$$
\conesupp (a) \cap WF(u) = \emptyset,
$$
then $a^w(x,D) u \in \cS$.
\end{prop}

\begin{cor}\label{emptyWFchar}
Let $u \in \cS'(\rr d)$. Then $WF(u) = \emptyset$ if and only if $u \in \cS (\rr d)$.
\end{cor}
\begin{proof}
If $WF(u) = \emptyset$ then by Proposition \ref{conesuppwavefront} $a^w(x,D) u \in \cS$ for any $a \in G^\infty$, in particular $a=1$ which gives $u \in \cS$.
On the other hand, if $u \in \cS(\rr d)$ then $a^w(x,D) u \in \cS$ for any $a \in G^\infty$. Since we may choose $a$ such that $z \notin \charac (a)$ for any given $z \in T^*(\rr d) \setminus \{ (0,0) \}$, it follows that $WF(u) = \emptyset$.
\end{proof}

\section{Rapid decay of the short-time Fourier transform in a cone}\label{rapidcone}

In this section we first introduce a new global wave front set, defined in terms of rapid decay of the STFT in conical sets, and denoted by $WF'(u)$.

\begin{defn}\label{wavefront2}
If $u \in \mathscr S'(\rr d)$ and $\varphi \in \mathscr S(\rr d) \setminus \{0 \}$ then for $z_0 \in T^*(\rr d) \setminus \{ (0,0) \}$ we say that $z_0 \notin WF'(u)$ if there exists an open conic set $\Gamma_{z_0} \subseteq T^*(\rr d) \setminus \{ (0,0) \}$ containing $z_0$ such that
\begin{equation*}
\sup_{z \in \Gamma_{z_0}} \eabs{z}^N |V_\varphi u(z)| < \infty \quad \forall N \geq 0.
\end{equation*}
\end{defn}

It follows that $WF'(u)$ is a closed conic subset of $T^*(\rr d) \setminus \{ (0,0) \}$.

The main goal of this section is to show some invariance and preparatory results.
These will serve on the one hand to give a quick proof of the identity
$WF'(u) = WF(u)$ for $u \in \mathscr S'(\rr d)$ in Section \ref{wfequality},
and on the other hand for the main result in Section \ref{wfinclusion},
which concerns the first inclusion of Proposition \ref{wavefrontinclusion1} for symbols $a \in S_{0,0}^0$.
We also show that the rapid decay in a conic set in Definition \ref{wavefront2} may be relaxed to
discrete rapid decay in a conic set intersected with a lattice that generates a Gabor frame,
i.e. we prove $WF'(u) = WF_G(u)$.

\begin{prop}\label{convolutioninvariance}
Let $f$ be a measurable function that satisfies
\begin{equation}\label{polynomialbound1}
|f(x)| \lesssim \eabs{x}^{M}, \quad x \in \rr d,
\end{equation}
for some $M \geq 0$, suppose $x_0 \in \rr d \setminus \{ 0 \}$, and suppose there exists an open conic set
$\Gamma \subseteq \rr d \setminus \{ 0 \}$ containing $x_0$ such that
\begin{equation}\label{conedecay1}
\sup_{x \in \Gamma} \eabs{x}^N |f(x)| < \infty \quad \forall N \geq 0.
\end{equation}
If
\begin{equation}\label{L1intersection}
g \in \bigcap_{s \geq 0} L_{s}^1(\rr d)
\end{equation}
then for any open conic set $\Gamma' \subseteq \rr d \setminus \{ 0 \}$ such that $x_0 \in \Gamma'$ and
$\overline{\Gamma' \cap S_{d-1}} \subseteq \Gamma$, we have
\begin{equation}\label{conedecay2}
\sup_{x \in \Gamma'} \eabs{x}^N |f * g(x)| < \infty \quad \forall N \geq 0.
\end{equation}
\end{prop}

\begin{proof}
Let $\ep>0$.
We estimate and split the convolution integral as
\begin{equation*}
|f * g(x)|
\leq  \underbrace{\int_{\eabs{y} \leq \ep \eabs{x}} |f(x-y)| \, | g (y)| \, d y}_{:= I_1}
+ \underbrace{\int_{\eabs{y} > \ep \eabs{x}} |f(x-y)| \, | g (y)| \, d y}_{:= I_2}.
\end{equation*}
Consider $I_1$.
Since $\eabs{y} \leq \ep \eabs{x}$ we have $x-y \in \Gamma$ if $x \in \Gamma'$, $|x| \geq 1$, and $\ep$ is chosen sufficiently small.
Let $N \geq 0$ be arbitrary.
The assumptions \eqref{conedecay1} and \eqref{L1intersection} give
\begin{equation}\label{intuppsk1}
\begin{aligned}
I_1 & \lesssim \int_{\eabs{y} \leq \ep \eabs{x}} \eabs{x-y}^{-N} |g (y)| \, d y
\lesssim \eabs{x}^{-N} \int_{\rr d} \eabs{y}^{N} |g(y)| \, d y \\
& \lesssim  \eabs{x}^{-N}, \quad x \in \Gamma', \quad |x| \geq 1.
\end{aligned}
\end{equation}
Next we estimate $I_2$ using \eqref{polynomialbound1}:
\begin{equation}\label{intuppsk2}
\begin{aligned}
I_2 & \lesssim \int_{\eabs{y} > \ep \eabs{x}} \eabs{x-y}^{M} \, |g(y) | \, dy \\
& \lesssim \eabs{x}^{M} \int_{\eabs{y} > \ep \eabs{x}}  \eabs{y}^{-M-N} \, \eabs{y}^{2M+N} \, |g(y) | \, dy \\
& \lesssim \eabs{x}^{M-M-N} \int_{\rr d} \eabs{y}^{2M+N} \, |g(y) | \, dy \\
& \lesssim \eabs{x}^{-N}, \quad x \in \rr d,
\end{aligned}
\end{equation}
again using \eqref{L1intersection}.
A combination of \eqref{intuppsk1} and \eqref{intuppsk2} proves \eqref{conedecay2}.
\end{proof}

The following corollary says that the condition of rapid decay of the STFT in an open cone in the phase space,
containing a given nonzero vector,
does not depend on the window function taken in $\mathscr S \setminus \{ 0 \}$.

\begin{cor}\label{windowinvariance}
Let $u \in \mathscr S'(\rr d)$, $\varphi \in \mathscr S(\rr d) \setminus \{ 0 \}$ and $z_0 \in T^*(\rr d) \setminus \{ (0,0) \}$.
Suppose there exists an open conic set $\Gamma \subseteq T^*(\rr d) \setminus \{ (0,0) \}$ containing $z_0$ such that
\begin{equation}\label{conedecayassumption1}
\sup_{z \in \Gamma} \eabs{z}^N |V_\varphi u(z)| < \infty \quad \forall N \geq 0.
\end{equation}
Then for any open conic set $\Gamma' \subseteq T^*(\rr d) \setminus \{ (0,0) \}$ such that $z_0 \in \Gamma'$ and
$\overline{\Gamma' \cap S_{2d-1}} \subseteq \Gamma$,
and any $\psi \in \mathscr S(\rr d) \setminus \{0\}$, we have
\begin{equation*}
\sup_{z \in \Gamma'} \eabs{z}^N |V_\psi u(z)| < \infty \quad \forall N \geq 0.
\end{equation*}
\end{cor}
\begin{proof}
By \cite[Lemma 11.3.3]{Grochenig1} we have for $z \in \rr {2d}$
\begin{equation*}
|V_\psi u (z)| \leq (2 \pi )^{-d} \, \| \varphi \|_{L^2}^{-2} \, |V_\varphi u| * |V_\psi \varphi |(z), \quad \psi \in \cS(\rr d) \setminus \{0\},
\end{equation*}
and by \cite[Theorem 11.2.3]{Grochenig1} we have for some $M \geq 0$
$$
|V_\varphi u (z)| \lesssim \eabs{z}^{M}, \quad z \in \rr {2d}.
$$
Since
$$
V_\psi \varphi \in \cS(\rr {2d}) \subseteq \bigcap_{s \geq 0} L_{s}^1(\rr {2d}),
$$
the result follows from Proposition \ref{convolutioninvariance}.
\end{proof}

By Corollary \ref{windowinvariance}, $WF'(u)$ does not depend on the window function $\varphi \in \mathscr S \setminus \{0\}$.

We now restate the definition \eqref{GaborWFdef} of the Gabor wave front set $WF_G(u)$, given in the Introduction.
Here we require rapid decay not in a conic open set containing a given phase space direction $z_0$,
but instead rapid decay only in such a conic set intersected with a lattice that together with the window function generates a Gabor frame.

\begin{defn}\label{wavefront3}
Let $\varphi \in \mathscr S(\rr d) \setminus \{0 \}$,
and define the lattice $\Lam=\alpha \zz d \times \beta \zz d \subseteq \rr {2d}$ where $\alpha, \beta>0$ are sufficiently small to ensure that $\{ \Pi(\lam) \fy \}_{\lam \in \Lam}$ is a Gabor frame for $L^2(\rr d)$.
If $u \in \mathscr S'(\rr d)$
and $z_0 \in T^*(\rr d) \setminus \{ (0,0) \}$ we say that $z_0 \notin WF_G(u)$ if there exists an open conic set $\Gamma_{z_0} \subseteq T^*(\rr d) \setminus \{ (0,0) \}$ containing $z_0$ such that
\begin{equation*}
\sup_{\lam \in \Lam \cap \Gamma_{z_0}} \eabs{\lam}^N |V_\varphi u(\lam)| < \infty \quad \forall N \geq 0.
\end{equation*}
\end{defn}

The following result shows that the wave front sets in Definition \ref{wavefront2} and Definition \ref{wavefront3} are equal.

\begin{thm}\label{gaborwavefront}
If $u \in \mathscr S'(\rr d)$  then $WF'(u) = WF_G(u)$.
\end{thm}
\begin{proof}
The inclusion $WF_G(u) \subseteq WF'(u)$ is trivial, so we need only prove
\begin{equation}\label{STFTGabor1}
WF'(u) \subseteq WF_G(u).
\end{equation}
Suppose $0 \neq z_0 \notin WF_G(u)$. Then there exists
an open conic set $\Gamma  \subseteq T^*(\rr d) \setminus \{ (0,0) \}$ containing $z_0$ such that
\begin{equation}\label{discretedecay1}
\sup_{\lam \in \Lam \cap \Gamma} \eabs{\lam}^N |V_\varphi u(\lam)| < \infty \quad \forall N \geq 0.
\end{equation}
By \eqref{sprimemodulation} and \eqref{Gaborreconweak} we have
\begin{equation*}
(u,g) = \sum_{\lam \in \Lam} V_\fy u (\lam) \, (\Pi(\lam)  \wt \fy,g), \quad g \in \cS(\rr d),
\end{equation*}
where $\wt \fy = S^{-1} \fy \in \cS$. If we define
\begin{align*}
u_1 & = \sum_{\lam \in \Lam \cap \Gamma} V_\fy u (\lam) \, \Pi(\lam)  \wt \fy, \\
u_2 & = \sum_{\lam \in \Lam \setminus \Gamma} V_\fy u (\lam) \, \Pi(\lam)  \wt \fy,
\end{align*}
then $V_\varphi u(z) = V_\varphi u_1(z)  + V_\varphi u_2(z)$.
For $\alpha,\beta \in \nn d$ arbitrary and $\lambda=(\lambda_1,\lambda_2)$ with $\lambda_1,\lambda_2 \in \rr d$, \eqref{discretedecay1} gives
\begin{align*}
& \left| x^\beta \pd \alpha u_1 (x) \right| \leq \sum_{\lam \in \Lam \cap \Gamma} |V_\fy u (\lam)| \, \left| x^\beta \pd \alpha \Pi(\lam)  \wt \fy (x) \right| \\
& \lesssim \sum_{\lam \in \Lam \cap \Gamma} |V_\fy u (\lam)| \, \sum_{\gamma \leq \alpha} \binom{\alpha}{\gamma} \eabs{x-\lam_1}^{|\beta|} \eabs{\lam_1}^{|\beta|}
\left| (i \lam_2)^{\alpha-\gamma} e^{i \la \lam_2,x \ra} \pd \gamma \wt \fy (x-\lam_1) \right| \\
& \lesssim \sum_{\lam \in \Lam \cap \Gamma} |V_\fy u (\lam)| \eabs{\lam}^{|\alpha|+|\beta| + 2 d + 1 - 2 d - 1}
\lesssim \sum_{\lam \in \Lam} \eabs{\lam}^{- 2 d - 1} < \infty, \quad x \in \rr d.
\end{align*}
It follows that $u_1 \in \cS(\rr d)$ and hence $V_\varphi u_1 \in \cS(\rr {2d})$.

Let $\Gamma' \subseteq T^*(\rr d) \setminus \{ (0,0) \}$ be an open conic set such that $z_0 \in \Gamma'$ and $\overline{ \Gamma' \cap S_{2d-1} } \subseteq \Gamma$.
Then
$$
\inf_{0 \neq \lam \in \Lam \setminus \Gamma, \, z \in \Gamma'} \left| \frac{\lam}{|\lam|} - z \right| = \ep>0,
$$
and therefore $| \lam - z | \geq \ep |\lam|$ holds when $0 \neq \lam \in \Lam \setminus \Gamma$ and $z \in \Gamma'$.
Since $V_\fy \wt \fy \in \cS(\rr {2d})$ and
$$
| V_\fy u(z) | \lesssim \eabs{z}^M, \quad z \in \rr {2d},
$$
for some $M \geq 0$, this gives for $N \geq 0$ arbitrary and $z \in \Gamma'$
\begin{align*}
\eabs{z}^N \left| V_\varphi u_2(z) \right|
& \lesssim \sum_{\lam \in \Lam \setminus \Gamma} |V_\fy u (\lam)| \, \eabs{\lam}^N \eabs{z-\lam}^N \left| (\Pi(\lam)  \wt \fy,\Pi(z) \fy) \right| \\
& \lesssim \sum_{\lam \in \Lam \setminus \Gamma} \eabs{\lam}^{M+N} \eabs{z-\lam}^N \left| V_\fy \wt \fy (z-\lam) \right| \\
& \lesssim \sum_{\lam \in \Lam \setminus \Gamma} \eabs{\lam}^{M+N} \eabs{z-\lam}^{N-2N-M-2d-1} \\
& \lesssim \sum_{\lam \in \Lam} \eabs{\lam}^{M+N-N-M-2d - 1} < \infty.
\end{align*}
Since we have already proved $V_\varphi u_1 \in \cS(\rr {2d})$, it follows that
$$
\sup_{z \in \Gamma'} \eabs{z}^N \left| V_\varphi u(z) \right| < \infty \quad \forall N \geq 0,
$$
i.e. $z_0 \notin WF'(u)$, which concludes the proof of \eqref{STFTGabor1}.
\end{proof}

A combination of Theorem \ref{gaborwavefront} with Corollary \ref{windowinvariance} shows that $WF_G(u)$ does not depend on
the pair $(\fy,\Lam)$ of window function in $\cS \setminus \{ 0 \}$ and lattice defined by $\alpha,\beta>0$ as long as these generate a Gabor frame for $L^2(\rr d)$.

We proceed with the preparatory results in order to show $WF'(u) = WF(u)$ for $u \in \mathscr S'(\rr d)$ in Section \ref{wfequality}, and Theorem \ref{wavefrontinclusion3} in Section \ref{wfinclusion}.
Let $a(x,D)$ be the Kohn--Nirenberg quantization of $a \in S_{0,0}^m$,
and let $\fy \in \cS(\rr d)$ with $\| \fy \|_{L^2}=1$.
Consider the kernel $K(y',\eta'; y,\eta) \in \cS'(\rr {4d})$ of the linear operator
\begin{equation}\label{operatorphase}
V_\fy \, a(x,D) V_\fy^*: \cS(\rr {2d}) \mapsto \cS'(\rr {2d}),
\end{equation}
so that
\begin{equation}\label{kernelrepr1}
V_\fy ( a(x,D) u ) (y',\eta') = \int_{\rr {2d}} K(y',\eta'; y,\eta) \, V_\fy u(y,\eta) \, d y \, d \eta.
\end{equation}
Direct computations (see for example \cite[pp.~58--59]{Nicola1}) show that we may express the kernel $K$ by the integral
\begin{equation}\label{kernel1}
\begin{aligned}
K(y',\eta'; y,\eta) & = (2 \pi )^{-2d} \, e^{i \la y, \eta \ra} \\
& \qquad \times \int_{\rr {2d}} e^{i \left( \la x,\xi \ra - \la y, \xi \ra - \la x, \eta' \ra \right)} \, a(x,\xi) \, \wh \fy (\xi-\eta) \, \overline{\fy(x-y')} \, dx \, d\xi, \\
& \qquad \qquad y',\eta', y,\eta \in \rr d.
\end{aligned}
\end{equation}

\begin{prop}\label{kerneldecay1}
Let $a \in S_{0,0}^0$ and let $K \in C^\infty(\rr{4d})$ be defined by \eqref{kernel1}. For every integer $N \geq 0$ we have
\begin{equation}\label{kernelestimate1}
| K(y',\eta'; y,\eta) | \lesssim \eabs{y-y'}^{-2N} \eabs{\eta-\eta'}^{-2N}, \quad (y',\eta'; y,\eta) \in \rr {4d}.
\end{equation}
Moreover, if $m \in \ro$, $a \in S_{0,0}^m$ and $a(z)=0$ for $z \in \Gamma \setminus \overline{B_{2d}^R}$ where $\Gamma \subseteq T^*(\rr d) \setminus \{ (0,0) \}$ is open and conic and $R>0$, then for every open conic set $\Gamma' \subseteq T^*(\rr d) \setminus \{ (0,0) \}$ such that $\overline{\Gamma' \cap S_{2d-1}} \subseteq \Gamma$, we have for all integers $N \geq 0$
\begin{equation}\label{kernelestimate2}
\begin{aligned}
| K(y',\eta'; y,\eta) |
& \lesssim \eabs{y-y'}^{-2N} \eabs{\eta-\eta'}^{-2N} \eabs{y'}^{-4N} \eabs{\eta'}^{-4N}, \\
& \qquad (y',\eta') \in \Gamma', \quad (y,\eta) \in \rr {2d}.
\end{aligned}
\end{equation}
\end{prop}

\begin{rem}
The estimates \eqref{kernelestimate1}, which we shall recapture in the subsequent proof, are actually well known, see for example \cite{Grochenig2}, \cite[pp.~58--59]{Nicola1}, \cite{Rochberg1}.
\end{rem}

\begin{proof}
By a linear change of variables in \eqref{kernel1} we obtain
\begin{equation*}
\begin{aligned}
& |K(y',\eta'; y,\eta)| \\
& = (2 \pi )^{-2d}
\left| \int_{\rr {2d}} e^{i \left( \la x,\eta-\eta'+\xi \ra + \la \xi, y'-y \ra \right)} \, a(x+y',\xi+\eta) \, \wh \fy (\xi) \, \overline{\fy(x)} \, dx \, d\xi \right|.
\end{aligned}
\end{equation*}
Writing for any $N,M \in \no$
\begin{align*}
& e^{i \left( \la x,\eta-\eta'+\xi \ra + \la \xi, y'-y \ra \right)} \\
& = \eabs{\eta-\eta'+\xi}^{-2M} (1-\Delta_x)^M e^{i \left( \la x,\eta-\eta'+\xi \ra + \la \xi, y'-y \ra \right)} \\
& = \eabs{y-y'}^{-2N} \eabs{\eta-\eta'+\xi}^{-2M} (1-\Delta_x)^M
e^{i \left( \la x,\eta-\eta'+\xi \ra \right)} (1-\Delta_\xi)^N e^{i \left( \la \xi, y'-y \ra \right)}
\end{align*}
and integrating by parts we have
\begin{equation}\label{kernelestimate3}
\begin{aligned}
& |K(y',\eta'; y,\eta)| \\
& = (2 \pi )^{-2d} \eabs{y-y'}^{-2N}
\left| \int_{\rr {2d}} e^{i \left( \la x,\eta-\eta' \ra + \la \xi, y'-y \ra \right)} \, \lambda_{N,M} (y',\eta',\eta,x,\xi) \, dx \, d\xi \right|
\end{aligned}
\end{equation}
where
\begin{align*}
& \lambda_{N,M} (y',\eta',\eta,x,\xi) \\
& = (1-\Delta_\xi)^N \left( e^{i \la x,\xi \ra} \eabs{\eta-\eta'+\xi}^{-2M} (1-\Delta_x)^M
\left(  a(x+y',\xi+\eta) \, \wh \fy (\xi) \, \overline{\fy(x)} \right) \right).
\end{align*}
For $a \in S_{0,0}^m$ and $N,M \in \no$ arbitrary we have the estimates
\begin{equation}\label{lambdaestimate}
\left| \lambda_{N,M} (y',\eta',\eta,x,\xi) \right|
\lesssim \eabs{\eta-\eta'+\xi}^{-2M} \eabs{\xi+\eta}^m \, \eabs{x}^{-k} \, \eabs{\xi}^{-k}
\end{equation}
for any $k \geq 0$, since $\fy \in \cS$.
Using
\begin{equation*}
\eabs{\eta-\eta'+\xi}^{-2M} \lesssim \eabs{\eta-\eta'}^{-2M} \eabs{\xi}^{2M},
\end{equation*}
and picking $k>d+2M+|m|$, the estimate
\begin{equation}\label{kernelestimate3b}
|K(y',\eta'; y,\eta)|
\lesssim \eabs{y-y'}^{-2N} \eabs{\eta-\eta'}^{|m|-2M} \eabs{\eta'}^{|m|}, \quad y',\eta',y,\eta \in \rr d,
\end{equation}
is now easily proved by substitution of \eqref{lambdaestimate} into \eqref{kernelestimate3}.
At this point \eqref{kernelestimate1} follows inserting $m=0$ and $M=N$.

Next we prove \eqref{kernelestimate2} under the assumptions $m \in \ro$, $a \in S_{0,0}^m$, $a(z)=0$ for $z \in \Gamma \setminus \overline{B_{2d}^R}$ where $\Gamma \subseteq T^*(\rr d) \setminus \{ (0,0) \}$ is open and conic, $R>0$, and $\Gamma' \subseteq T^*(\rr d) \setminus \{ (0,0) \}$ is open conic such that $\overline{\Gamma' \cap S_{2d-1}} \subseteq \Gamma$.
We continue from \eqref{kernelestimate3} and \eqref{lambdaestimate} using the support properties of $a$.
Thus for $N,M \in \no$ arbitrary we have
\begin{equation}\label{kernelestimate4}
\begin{aligned}
|K(y',\eta'; y,\eta)|
& \lesssim \eabs{y-y'}^{-2N} \eabs{\eta-\eta'}^{-M} \\
& \quad \times \int_{D_{y',\eta}}
\eabs{\eta'-(\xi+\eta)}^{-M} \, \eabs{\xi+\eta}^m \, \eabs{x}^{-k} \, \eabs{\xi}^{-k+M}
dx \, d\xi
\end{aligned}
\end{equation}
where $k \geq 0$ is arbitrary, and
$$
D_{y',\eta} = \{ (x,\xi) \in \rr {2d}: \, (x+y',\xi+\eta) \in (\rr {2d} \setminus \Gamma) \cup \overline{B_{2d}^R} \}.
$$
We want to estimate $|K(y',\eta'; y,\eta)|$ for $(y',\eta') \in \Gamma'$ and $(y,\eta) \in \rr {2d}$.
There exists $\ep>0$ that does not depend on $x,\xi,y',\eta',y,\eta$ such that
\begin{equation}\label{lowerbound1}
\begin{aligned}
& \left| \frac{(y',\eta')}{|(y',\eta')|} - \frac{(x+y',\xi+\eta)}{|(y',\eta')|} \right| \geq \ep, \\
& \qquad (y',\eta') \in \Gamma', \quad |(y',\eta')| \geq 2R, \quad (x,\xi) \in D_{y',\eta}, \quad (y,\eta) \in \rr {2d}.
\end{aligned}
\end{equation}
In fact, for $(x,\xi) \in D_{y',\eta}$ such that $(x+y',\xi+\eta) \in \overline{B_{2d}^R}$ we have
\begin{equation*}
\left| \frac{(y',\eta')}{|(y',\eta')|} - \frac{(x+y',\xi+\eta)}{|(y',\eta')|} \right|
\geq 1 - \frac{R}{2R} = \frac1{2},
\end{equation*}
and for $(x,\xi) \in D_{y',\eta}$ such that $(x+y',\xi+\eta) \in \rr {2d} \setminus \Gamma$,
the assumption $\overline{\Gamma' \cap S_{2d-1}} \subseteq \Gamma$ implies \eqref{lowerbound1} for some $\ep>0$.
Hence
\begin{align*}
& \left|(x,\eta'-(\xi+\eta)) \right| \geq \ep |(y',\eta')|, \\
& \qquad \qquad (y',\eta') \in \Gamma', \quad |(y',\eta')| \geq 2R, \quad (x,\xi) \in D_{y',\eta}, \quad (y,\eta) \in \rr {2d},
\end{align*}
which gives
\begin{equation*}
\eabs{y'} \eabs{\eta'} \lesssim \eabs{x}^2 \eabs{\eta'-(\xi+\eta)}^2
\end{equation*}
when $(y',\eta') \in \Gamma'$, $|(y',\eta')| \geq 2R$, $(x,\xi) \in D_{y',\eta}$, $(y,\eta) \in \rr {2d}$.
For $(y',\eta') \in \Gamma'$, $|(y',\eta')| \geq 2R$ and $(y,\eta) \in \rr {2d}$ we obtain
\begin{equation*}
\begin{aligned}
& \int_{D_{y',\eta}} \eabs{\eta'-(\xi+\eta)}^{-M} \, \eabs{\xi+\eta}^m \, \eabs{x}^{-k} \, \eabs{\xi}^{-k+M}
dx \, d\xi \\
& \lesssim \eabs{y'}^{-M/2} \eabs{\eta'}^{-M/2} \eabs{\eta}^{|m|} \int_{\rr {2d}} \eabs{x}^{M-k} \eabs{\xi}^{|m|-k+M} dx \, d\xi \\
& \lesssim \eabs{y'}^{-M/2} \eabs{\eta'}^{-M/2+|m|} \eabs{\eta-\eta'}^{|m|}
\end{aligned}
\end{equation*}
provided $k>M+d+|m|$.
Insertion into \eqref{kernelestimate4} gives
\begin{equation*}
|K(y',\eta'; y,\eta)|
\lesssim \eabs{y-y'}^{-2N} \eabs{\eta-\eta'}^{-M+|m|} \eabs{y'}^{-M/2} \eabs{\eta'}^{-M/2+|m|}.
\end{equation*}
If we pick $M \geq 8 N + 2 |m|$ then $-M+|m| \leq - 2 N$ and $-M/2 \leq -M/2+|m| \leq - 4 N$,
which finally proves \eqref{kernelestimate2},
since the estimate \eqref{kernelestimate2} for $|(y',\eta')| \leq 2R$ follows from \eqref{kernelestimate3b}.
\end{proof}

\begin{cor}\label{WFconesupp}
If $m \in \ro$, $a \in S_{0,0}^m$ and $u \in \cS'(\rr d)$ then
\begin{equation*}
WF'( a(x,D) u ) \subseteq \conesupp (a).
\end{equation*}
\end{cor}

\begin{proof}
Suppose $0 \neq z_0 \notin \conesupp (a)$.
This means that there exists an open conic set $\Gamma \subseteq T^* (\rr d) \setminus \{ (0,0) \}$ with $z_0 \in \Gamma$, such that $a(z) = 0$ for $z \in \Gamma \setminus \overline{B_{2d}^R}$ for some $R>0$.
Let us prove $z_0 \notin WF'( a(x,D) u )$ by considering $V_\fy (a(x,D) u)$ with $\fy \in \cS(\rr d)$ and $\| \fy \|_{L^2}=1$.
The assumptions of the second part of Proposition \ref{kerneldecay1} are satisfied.
Hence for any open conic $\Gamma' \subseteq T^*(\rr d) \setminus \{ (0,0) \}$ such that $z_0 \in \Gamma'$ and $\overline{\Gamma' \cap S_{2d-1}} \subseteq \Gamma$ we have from \eqref{kernelrepr1} and
\eqref{kernelestimate2}, for any integer $N \geq 0$,
\begin{equation}\label{kernelestimate5}
\begin{aligned}
& \left| V_\fy ( a(x,D) u ) (y',\eta')  \right| \\
& \lesssim \eabs{y'}^{-4N} \eabs{\eta'}^{-4N} \int_{\rr {2d}}
\eabs{y-y'}^{-2N} \eabs{\eta-\eta'}^{-2N} \, |V_\fy u(y,\eta)| \, d y \, d \eta, \\
& \qquad (y',\eta') \in \Gamma'.
\end{aligned}
\end{equation}
Using the fact that $|V_\fy u(y,\eta)| \lesssim \eabs{(y,\eta)}^k$ for some $k \geq 0$ and for all $(y,\eta) \in \rr {2d}$, it follows from \eqref{kernelestimate5} with $N$ sufficiently large
\begin{equation*}
\begin{aligned}
\left| V_\fy ( a(x,D) u ) (y',\eta')  \right|
& \lesssim \eabs{y'}^{-2N} \eabs{\eta'}^{-2N} \\
& \leq \eabs{(y',\eta')}^{-2N}, \quad (y',\eta') \in \Gamma'.
\end{aligned}
\end{equation*}
It follows that $V_\fy ( a(x,D) u )$ has rapid decay in $\Gamma'$ which proves $z_0 \notin WF'( a(x,D) u)$.
\end{proof}

\begin{rem}\label{KohnNirenbergequivdef2}
The statements of Proposition \ref{kerneldecay1} and Corollary \ref{WFconesupp} are valid also if we consider the Weyl quantization instead of the Kohn--Nirenberg quantization.
In fact, if $a(x,D)=b^w(x,D)$ then $a \in S_{0,0}^m$ if and only if $b \in S_{0,0}^m$, see e.g. \cite[Theorem~2.37]{Folland1}.
The symbols $a$ and $b$ are related by (see \cite[Chapter~18.5]{Hormander0})
\begin{equation*}
a(x,\xi)
= e^{\frac{i}{2} \la D_x, D_\xi \ra } b(x,\xi)
= \pi^{-d} \int_{\rr {2d}} e^{-2 i \la y, \eta \ra} b(x+y,\xi+\eta) \, d y \, d \eta
\end{equation*}
which is interpreted not as an oscillatory integral (because the symbol class $S_{0,0}^m$ does not admit those), but instead as a Fourier multiplier operator in $\cS'(\rr {2d})$.
Regularizing $b \in S_{0,0}^m$ as $b_\ep = b \chi_\ep$ where $\chi \in C_c^\infty(\rr {2d})$ equals one in a neighborhood of the origin and $\chi_\ep(z) = \chi (\ep z)$, we have $b_\ep \rightarrow b$ and $a_\ep \rightarrow a$ in $\cS'$ as $\ep \rightarrow +0$, where $a_\ep(x,\xi) = e^{\frac{i}{2} \la D_x, D_\xi \ra } b_\ep(x,\xi)$.

Assume that $b(z)=0$ for $z \in \Gamma \setminus \overline{B_{2d}^R}$ where $\Gamma \subseteq T^*(\rr d) \setminus \{ (0,0) \}$ is open, conic and $R>0$, and let $\Gamma' \subseteq T^*(\rr d) \setminus \{ (0,0) \}$ be open and conic such that $\overline{\Gamma' \cap S_{2d-1}} \subseteq \Gamma$.
Introduce conic open sets $\Gamma_1, \Gamma_1' \subseteq T^*(\rr d) \setminus \{ (0,0) \}$ such that
$\overline{\Gamma' \cap S_{2d-1}} \subseteq \Gamma_1'$, $\overline{\Gamma_1' \cap S_{2d-1}} \subseteq \Gamma_1$, and
$\overline{\Gamma_1 \cap S_{2d-1}} \subseteq \Gamma$.
Integration by parts and
\begin{equation*}
e^{-2 i \la y, \eta \ra} = \eabs{y}^{- 2 N} \left(1-\frac1{4} \Delta_\eta\right)^N \eabs{\eta}^{- 2 N} \left(1-\frac1{4} \Delta_y \right)^N e^{-2 i \la y, \eta \ra}
\end{equation*}
for any $N \in \no$, yield
\begin{equation*}
\begin{aligned}
a_\ep(x,\xi)
& = \pi^{-d} \int_{\rr {2d}} e^{-2 i \la y, \eta \ra}
\eabs{y}^{- 2 N} \\
& \qquad \times \left(1-\frac1{4} \Delta_\eta\right)^N
\left( \eabs{\eta}^{- 2 N} \left(1-\frac1{4} \Delta_y \right)^N b_\ep(x+y,\xi+\eta) \right) \,
d y \, d \eta.
\end{aligned}
\end{equation*}
Provided $N>(|m|+d)/2$ we have by dominated convergence
\begin{equation}\label{symboldecay1}
\begin{aligned}
a(x,\xi)
& = \pi^{-d} \int_{D_{x,\xi}} e^{-2 i \la y, \eta \ra}
\eabs{y}^{- 2 N} \\
& \qquad \times \left(1-\frac1{4} \Delta_\eta\right)^N
\left( \eabs{\eta}^{- 2 N} \left(1-\frac1{4} \Delta_y \right)^N b(x+y,\xi+\eta) \right) \,
d y \, d \eta
\end{aligned}
\end{equation}
where
$$
D_{x,\xi} = \{ (y,\eta) \in \rr {2d}: \, (x+y,\xi+\eta) \in (\rr {2d} \setminus \Gamma) \cup \overline{B_{2d}^R} \}.
$$
As in the proof of Proposition \ref{kerneldecay1} we have
\begin{equation*}
\eabs{(y,\eta)} \geq \ep \eabs{(x,\xi)}, \quad (x,\xi) \in \Gamma_1, \quad |(x,\xi)| \geq 2R, \quad (y,\eta) \in D_{x,\xi},
\end{equation*}
for some $\ep>0$ that does not depend on $x,\xi,y,\eta$. Inserted into \eqref{symboldecay1} this gives, for any $M \geq 0$ and $\alpha,\beta \in \nn d$,
\begin{equation*}
\begin{aligned}
\eabs{(x,\xi)}^M | \pdd x \alpha \pdd \xi \beta a(x,\xi) |
& \lesssim \int_{D_{x,\xi}} \eabs{y}^{- 2 N} \eabs{\eta}^{- 2 N} \eabs{(x,\xi)}^M \eabs{\xi+\eta}^m \, d y \, d \eta \\
& \lesssim \int_{D_{x,\xi}} \eabs{(y,\eta)}^{- 2N + |m|}\eabs{(x,\xi)}^{M+|m|} \, d y \, d \eta \\
& \lesssim \int_{\rr {2d}} \eabs{(y,\eta)}^{- 2N + 2|m| + M} \, d y \, d \eta \\
& \lesssim 1, \quad (x,\xi) \in \Gamma_1,
\end{aligned}
\end{equation*}
provided $N > |m|+d+M/2$. Thus $a$ behaves like a Schwartz function in $\Gamma_1$.
Let $\chi \in G^0$ satisfy $\supp(\chi) \subseteq \Gamma_1$ and $\chi(z)=1$ for $z \in \Gamma_1' \setminus B_{2d}^1$. Then $a_1 = a (1-\chi) \in S_{0,0}^m$ is zero in $\Gamma_1' \setminus \overline{B_{2d}^1}$ while $a_2 = a - a_1 = a \chi \in \cS(\rr {2d})$.
Hence we may decompose the kernel $K$ of $V_\fy \, a(x,D) V_\fy^* = V_\fy \, b^w(x,D) V_\fy^*$, where $\fy \in \cS$ and $\| \fy \|_{L^2}=1$, as $K=K_1+K_2$ where $K_1$ corresponds to $a_1$ and $K_2$ corresponds to $a_2$.
From \eqref{kernel1} and $a_2 \in \cS(\rr {2d})$ it follows that $K_2 \in \cS(\rr {4d})$.
It now follows that we may replace $a(x,D)$ with $a^w(x,D)$ in Proposition \ref{kerneldecay1}.
More precisely we replace $K$ defined by \eqref{kernel1} by the kernel of the operator \eqref{operatorphase} with $a(x,D)$ replaced by $a^w(x,D)$.

Also Corollary \ref{WFconesupp} holds for the Weyl quantization, in the sense that
\begin{equation}\label{wavefrontconesupp1}
WF'( b^w(x,D) u ) \subseteq \conesupp (b)
\end{equation}
for $m \in \ro$, $b \in S_{0,0}^m$ and $u \in \cS'(\rr d)$.
In fact, let $0 \neq z_0 \notin \conesupp (b)$. There exists an open conic set $\Gamma \subseteq T^*(\rr d) \setminus \{(0,0) \}$ such that $z_0 \in \Gamma$ and $b(z)=0$ for $z \in \Gamma \setminus \overline{B_{2d}^R}$ for some $R>0$.
With $a(x,D)=b^w(x,D)$ we have $a \in S_{0,0}^m$,
and as above we have $a=a_1+a_2$ where $a_2 \in \cS(\rr {2d})$ and $z_0 \notin \conesupp (a_1)$.
It now follows from Corollary \ref{WFconesupp} that $z_0 \notin WF'( a_1(x,D) u )$, and $a_2(x,D) u \in \cS$ gives finally $z_0 \notin WF'( a(x,D) u ) = WF'( b^w(x,D) u )$.
\end{rem}

We finish this section with a small observation which is another corollary of Proposition \ref{kerneldecay1}.
For the Schwartz kernel $H \in \cS'(\rr {2d})$ of a pseudodifferential operator $a(x,D)$ (or $a^w(x,D)$) with symbol in $a \in S_{0,0}^0$, the corollary shows that the wave front set $WF'(H) \subseteq T^*(\rr {2d}) \setminus \{ (0,0) \}$
is contained in the
union of $(\Delta \times \rr {2d}) \setminus \{0\}$ and $(\rr {2d} \times \Delta') \setminus \{0\}$,
where $\Delta$ and $\Delta'$ denote the diagonal and antidiagonal in $\rr d \times \rr d$, respectively:
$$
\Delta=\{(x,x) \in \rr {2d}: \, x \in \rr d \}, \quad
\Delta'=\{(\xi,-\xi) \in \rr {2d}: \, \xi \in \rr d \}.
$$
This is rather analogous to the classical wave front set of the Schwartz kernel of a pseudodifferential operator with symbol in $S_{\rho,\delta}^m$ with $0< \rho \leq 1$ and $0 \leq \delta < 1$, which is contained in the conormal bundle of the diagonal in $\rr d \times \rr d$ (cf. \cite[Theorem 18.1.16]{Hormander0}), i.e.
\begin{equation*}
\{ (x,x; \xi,-\xi), \, x,\xi \in \rr d \}.
\end{equation*}

\begin{cor}\label{WFkernel}
If $a \in S_{0,0}^0$ and $H \in \cS'(\rr {2d})$ is the Schwartz kernel of $a(x,D)$
then
$$
WF'(H) \subseteq \left( (\Delta \times \rr {2d}) \, \bigcup \, (\rr {2d} \times \Delta') \right) \setminus \{ 0 \}.
$$
\end{cor}

\begin{proof}
Let $\fy \in \cS(\rr d)$, $\| \fy \|_{L^2}=1$, and $F,G \in \cS(\rr {2d})$.
As before the kernel of $V_\fy a(x,D) V_\fy^*$ is denoted $K \in \cS'(\rr {4d})$.
We have
\begin{align*}
(K,G \otimes \overline{F}) & = (V_\fy a(x,D) V_\fy^*F,G)
= (2 \pi)^{d} (a(x,D) V_\fy^*F,V_\fy^*G) \\
& = (2 \pi)^{d} (H, V_\fy^*G \otimes \overline{V_\fy^*F} ).
\end{align*}
Since
\begin{align*}
& (V_\fy^*G \otimes \overline{V_\fy^*F}) (x,y) \\
& = (2 \pi)^{-2d} \int_{\rr {4d}} G(z) \, \overline{F(w)} \, \Pi(z) \fy (x) \, \overline{\Pi(w) \fy(y)} \, dz \, dw \\
& = (2 \pi)^{-2d} \int_{\rr {4d}} G(z_1,z_2) \, \overline{F(w_1,w_2)} \, \Pi(z_1,w_1; z_2,-w_2) \fy \otimes \overline{\fy}(x,y) \,\, dz \, dw \\
& = V_\Phi^* (T(G \otimes \overline{F})) (x,y),
\end{align*}
where $\Phi=\fy \otimes \overline{\fy}$ and $TU(z_1,z_2;w_1,w_2)= U(z_1,w_1;z_2,-w_2)$,
we may conclude
\begin{equation*}
V_\Phi H(z,w) = (2 \pi)^d TK(z,w) = (2 \pi)^d K(z_1,w_1;z_2,-w_2).
\end{equation*}
From Proposition \ref{kerneldecay1} it now follows that we have for any $N \geq 0$
\begin{equation*}
|V_\Phi H(z,w)| \lesssim \eabs{z_1-z_2}^{-2N} \eabs{w_1+w_2}^{-2N}, \quad z,w \in \rr {2d}.
\end{equation*}
If $(z_0,w_0) \notin \left( (\Delta \times \rr {2d}) \, \bigcup \, (\rr {2d} \times \Delta') \right) \setminus \{ 0 \}$
then $(z_0,w_0) \in \Gamma$ where $\Gamma \subseteq T^*(\rr {2d}) \setminus \{ 0 \}$ is the open conic set
\begin{equation*}
\Gamma = \{ (z,w) \in \rr {4d}: \, |z_1+z_2| < C|z_1-z_2|, \, |w_1-w_2| < C|w_1+w_2| \}
\end{equation*}
defined by a sufficiently large constant $C>0$.
From $2|z_1|^2 + 2 |z_2|^2 = |z_1+z_2|^2 + |z_1-z_2|^2$ it now follows that $V_\Phi H$ decays rapidly in $\Gamma$.
\end{proof}

\section{Equality of $WF(u)$ and $WF'(u)$}\label{wfequality}

H\"ormander \cite[Proposition~6.8]{Hormander1} proved the equality
$$
WF(u)=WF'(u), \quad u \in \cS'(\rr d),
$$
when $WF'(u)$ is defined by means of a Gaussian window function.
The proof in \cite{Hormander1} of the inclusion $WF'(u) \subseteq WF(u)$ is very complicated with all steps written out explicitly.
It is based on the ideas of a covering of $WF(u)$ by a finite union
of halfplanes of the form $H_w=\{z \in \rr {2d}: \la z, w \ra > 0 \}$ for $w \in S_{2d-1}$,
a conical partition of unity, symplectic invariance, reduction to $w=e_1$ (the first standard basis vector in $\rr {2d}$), and estimates of the STFT defined by a Gaussian window.

In this section we show that there is a much shorter and natural proof which is based on time-frequency analysis,
the results of Section \ref{rapidcone} and in particular Corollary \ref{WFconesupp}.
We show $WF'(u) = WF(u)$ by showing the two inclusions $WF'(u) \subseteq WF(u) \subseteq WF'(u)$.

\begin{thm}
If $u \in \cS'(\rr d)$ then
$$
WF'(u) \subseteq WF(u).
$$
\end{thm}

\begin{proof}
Suppose $0 \neq z_0 \notin WF(u)$. By Proposition \ref{noncharacteristic1} there exists $a \in G^0$ with $a(z)=1$ when $z \in \Gamma \setminus B_{2d}^1$ where $\Gamma \subseteq T^*(\rr d) \setminus \{(0,0) \}$ is open, conic and contains $z_0$, such that $a^w(x,D)u \in \cS$.
If $b=1-a$ then $b \in G^0$, $b(z)=0$ when $z \in \Gamma \setminus B_{2d}^1$ and
\begin{equation*}
u = b^w(x,D) u + a^w(x,D) u.
\end{equation*}
Since $a^w(x,D)u \in \cS$ we have
$WF'(u) = WF'( b^w(x,D)u)$.
On the other hand, from Corollary \ref{WFconesupp}, Remark \ref{KohnNirenbergequivdef2} and the observation $G^0 \subseteq S_{0,0}^0$:
\begin{equation*}
WF'( b^w(x,D) u ) \subseteq \conesupp (b).
\end{equation*}
Since $z_0 \notin \conesupp (b)$ we conclude $z_0 \notin WF'(u)$.
\end{proof}

\begin{thm}\label{decayimpliesnotWF}
If $u \in \mathscr S'(\rr d)$ then
$$
WF(u) \subseteq WF'(u).
$$
\end{thm}
\begin{proof}
Suppose $0 \neq z_0 \notin WF'(u)$.
Let $\psi(x)=\exp(-|x|^2/2)$ for $x \in \rr d$, which gives the Wigner distribution
$$
W(\psi)(z) = (4 \pi)^{d/2} \exp\left( - |z|^2 \right), \quad z \in \rr {2d}.
$$
There exists an open conic set $\Gamma \subseteq T^*(\rr d) \setminus \{ (0,0) \}$
such that $z_0 \in \Gamma$ and
\begin{equation}\label{conicdecay1}
\sup_{z \in \Gamma} \eabs{z}^N |V_\psi u(z)| < \infty \quad \forall N \geq 0.
\end{equation}
Let $a^w(x,D)$ be the localization operator (cf. \cite{Cordero1})
$$
a^w(x,D) u(x) = \int_{\rr {2d}} b(z) \ V_\psi u(z) \ \Pi(z) \psi(x) \ dz
$$
defined by the symbol $b$ and the window function $\psi$.
Then (cf. \cite{Cordero1} and \cite[Theorem 24.1]{Shubin1})
$$
a = b * W(\psi).
$$
Let $\Gamma' \subseteq \Gamma$  be an open conic set such that $z_0 \in \Gamma'$, $\overline{\Gamma' \cap S_{2d-1}} \subseteq \Gamma$.
Let $b \in G^0$ satisfy $0 \leq b \leq 1$, $\supp (b) \subseteq \Gamma$, and $b(z)=1$ when $z \in \Gamma'$ and $|z| \geq 1$.
Then $a \in C^\infty(\rr {2d})$ and we have for any $\alpha \in \nn d$
\begin{equation*}
\begin{aligned}
|\partial^\alpha a(z)|
& \leq \int_{\rr {2d}} |\partial^\alpha b(z-u)| \ W(\psi)(u) \ du \\
& \lesssim \int_{\rr {2d}} \eabs{z-u}^{-|\alpha|} \ W(\psi)(u) \ du \\
& \lesssim \eabs{z}^{-|\alpha|} \int_{\rr {2d}} \eabs{u}^{|\alpha|} \exp\left( - |u|^2 \right) \ du.
\end{aligned}
\end{equation*}
It follows that $a \in G^0$.

Let $\Gamma'' \subseteq \Gamma'$  be an open conic set such that $z_0 \in \Gamma''$, $\overline{\Gamma'' \cap S_{2d-1}} \subseteq \Gamma'$.
There exists $\delta>0$ such that $z-u/t \in \Gamma'$ provided $z \in \Gamma''$, $|z|=1$, $|u| \leq \delta$ and $t \geq 1$.
Moreover, $|z-u| \geq |z| - \delta \geq 1$ provided $|u| \leq \delta$ and $|z| \geq 1+\delta$.
These observations give for $z \in \Gamma''$ and $|z| \geq 1+\delta$
\begin{equation}\nonumber
\begin{aligned}
| a(z) | & = \int_{\rr {2d}} b(z - u) \ W(\psi)(u) \ du \\
& \geq \int_{|u| \leq \delta} b(|z|(z/|z|-u/|z|)) \ W(\psi)(u) \ du \\
& = \int_{|u| \leq \delta} \ W(\psi)(u) \ du > 0.
\end{aligned}
\end{equation}
Thus $z_0 \notin \charac (a)$.

Finally we prove $a^w(x,D) u \in \mathscr S$.
Using $\supp (b) \subseteq \Gamma$, $b \leq 1$, the assumption \eqref{conicdecay1}
and $\psi \in \mathscr S$,
we estimate a Schwartz space seminorm of $a^w(x,D) u$ as follows:
\begin{equation}\nonumber
\begin{aligned}
& |x^\beta \partial^\alpha a^w(x,D) u(x)|
\leq \iint_{\Gamma} |V_\psi u(t,\xi)| \ | x^\beta \partial_x^\alpha ( M_\xi T_t \psi)(x)  | \ dt \ d \xi \\
& \lesssim \sum_{\gamma \leq \alpha} \iint_{\Gamma} |V_\psi u(t,\xi)| \ | x^\beta \xi^{\alpha-\gamma} \partial^\gamma \psi(x-t)  | \ dt \ d \xi \\
& \lesssim \sum_{\gamma \leq \alpha} \iint_{\Gamma} |V_\psi u(t,\xi)| \ \eabs{t}^{|\beta|} \eabs{\xi}^{|\alpha|} \eabs{x-t}^{|\beta|} \ | \partial^\gamma \psi(x-t)  | \ dt \ d \xi \\
& \lesssim \iint_{\Gamma} \eabs{(t,\xi)}^{-|\alpha|-|\beta|- 2 d -1} \ \eabs{(t,\xi)}^{|\alpha|+|\beta|} \ dt \ d \xi
\leq C_{\alpha,\beta}
\end{aligned}
\end{equation}
for some $C_{\alpha,\beta} < \infty$, for all $x \in \rr d$. Since $\alpha,\beta \in \nn d$ are arbitrary, it follows that
$a^w(x,D) u \in \mathscr S$.
Now we have proved that there exists $a \in G^0$ such that $a^w(x,D) u \in \mathscr S$ and $z_0 \notin \charac (a)$, i.e. $z_0 \notin WF(u)$.
\end{proof}

\begin{cor}\label{notWFequalsdecay}
If $u \in \mathscr S'(\rr d)$ then $WF(u) =WF'(u) = WF_G(u) $.
\end{cor}

The corollary motivates the abolition of the notations $WF'(u)$ and $WF(u)$ in the rest of the paper, in favor of $WF_G(u)$, and allows to transfer to $WF_G(u)$ all the preceding properties of $WF'(u)$ and $WF(u)$.

\section{A wave front set inclusion for symbols in $S_{0,0}^0$}\label{wfinclusion}

\begin{thm}\label{wavefrontinclusion3}
If $a \in S_{0,0}^0$ and $u \in \cS'(\rr d)$ then
$$
WF_G( a^w(x,D) u) \subseteq WF_G(u) \bigcap \conesupp(a).
$$
\end{thm}
\begin{proof}
We have (see e.g. \cite{Holst1})
\begin{equation}\label{symbolintersection}
S_{0,0}^0 = \bigcap_{N \geq 0} M_{v_N}^{\infty,1}
\end{equation}
where $v_N$ is the weight $v_N(x,\xi) = \eabs{\xi}^N$ for $(x,\xi) \in \rr {2d} \oplus \rr {2d}$, and $M_{v_N}^{\infty,1} = M_{v_N}^{\infty,1}(\rr {2d})$
is a weighted modulation space.
The space $M_{v_N}^{\infty,1}$
is also known as a weighted version of Sj\"ostrand's symbol class (cf. \cite{Grochenig2,Sjostrand1,Sjostrand2}).

Let $\varphi \in \cS(\rr d)$ with $\| \varphi \|_{L^2}=1$.
Denoting the formal adjoint of $a^w(x,D)$ by $a^w(x,D)^*$, \eqref{STFTrecon} gives for $z \in \rr {2d}$
\begin{align*}
V_\varphi (a^w(x,D) u) (z)
& = ( a^w(x,D) u, \Pi(z) \varphi ) \\
& = ( u, a^w(x,D)^* \Pi(z) \varphi ) \\
& = (2 \pi)^{-d} \int_{\rr {2d}} V_\varphi u(w) \, ( \Pi(w) \varphi,a^w(x,D)^* \Pi(z) \varphi ) \, dw \\
& = (2 \pi)^{-d} \int_{\rr {2d}} V_\varphi u(w) \, ( a^w(x,D) \, \Pi(w) \varphi,\Pi(z) \varphi ) \, dw \\
& = (2 \pi)^{-d} \int_{\rr {2d}} V_\varphi u(z-w) \, ( a^w(x,D) \, \Pi(z-w) \varphi,\Pi(z) \varphi ) \, dw.
\end{align*}
By \eqref{symbolintersection} and \cite[Theorem 3.2]{Grochenig2}, for any $s \geq 0$ there exists $g_s \in L_{s}^1(\rr {2d})$ such that
\begin{align*}
\left| ( a^w(x,D) \, \Pi(z-w) \varphi,\Pi(z) \varphi ) \right| \leq g_s(w), \quad z, w \in \rr {2d}.
\end{align*}
With $g(w) = \sup_{z \in \rr {2d}} | ( a^w(x,D) \, \Pi(z-w) \varphi,\Pi(z) \varphi ) |$
we thus have
$$
g \in \bigcap_{s \geq 0} L_s^1(\rr {2d}),
$$
and
\begin{align}\label{convolution1}
|V_\varphi (a^w(x,D) u) (z)|
& \lesssim |V_\varphi u| * g(z), \quad z \in \rr {2d}.
\end{align}
If $0 \neq z_0 \in T^*(\rr d) \setminus WF_G(u)$ then there exists an open conic set $\Gamma \subseteq T^* (\rr d) \setminus \{ (0,0) \}$ containing $z_0$ such that
\begin{equation*}
\sup_{z \in \Gamma} \eabs{z}^N |V_\varphi u(z)| < \infty \quad \forall N \geq 0.
\end{equation*}
By \cite[Theorem 11.2.3]{Grochenig1} we have for some $M \geq 0$
$$
|V_\varphi u (z)| \lesssim \eabs{z}^{M}, \quad z \in \rr {2d}.
$$
It now follows from \eqref{convolution1} and Proposition \ref{convolutioninvariance}
that for any open conic set $\Gamma'$ containing $z_0$ such that $\overline{\Gamma' \cap S_{2d-1}} \subseteq \Gamma$ we have
\begin{equation*}
\sup_{z \in \Gamma'} \eabs{z}^N |V_\varphi (a^w(x,D) u) (z)| < \infty \quad \forall N \geq 0,
\end{equation*}
which proves that $z_0 \notin WF_G( a^w(x,D) u)$.
Thus we have shown
$$
WF_G( a^w(x,D) u) \subseteq WF_G(u).
$$
The remaining inclusion $WF_G( a^w(x,D) u) \subseteq \conesupp(a)$ is an immediate consequence of
Corollary \ref{WFconesupp}, Remark \ref{KohnNirenbergequivdef2} and Corollary \ref{notWFequalsdecay}.
\end{proof}

Since modulation and translation are invertible operators with Weyl symbols in $S_{0,0}^0$,
the result gives the following consequence.

\begin{cor}\label{shiftinvariant}
If $u \in \cS'(\rr d)$ and $z \in \rr {2d}$ then
$$
WF_G( \Pi(z) u ) = WF_G(u).
$$
\end{cor}

\section{Comparison with the $\cS$-wave front set}\label{swf}

Coriasco and Maniccia \cite{Coriasco1} have studied another type of global wave front set called the
$\cS$-wave front set which is adapted to the SG pseudodifferential calculus.
It includes as a component the classical wave front set.

To define the $\cS$-wave front set $WF_{\cS}(u)$ we need some concepts from \cite{Coriasco1}.
For $x_0 \in \rr d$, we denote by $\fy_{x_0}$ a function in $C_c^\infty(\rr d)$ such that $0 \leq \fy_{x_0} \leq 1$ and $\fy_{x_0}$ equals one in a neighborhood of $x_0$.
For $\xi_0 \in \rr d \setminus \{0\}$, we denote by $\psi_{\xi_0}$ a nonnegative function in $C^\infty(\rr d)$ supported in a conic open set $\Gamma \subseteq \rr d \setminus \{ 0 \}$ containing $\xi_0$, such that $\psi_{\xi_0}(\xi)=1$ when $\xi \in \Gamma'$ and $|\xi| \geq A$ for a conic open set $\Gamma' \subseteq \Gamma$,
and $\psi_{\xi_0}(\xi)=0$ for $\xi \in B_d^R$ for some $A > R > 0$.

\begin{defn}\label{wavefront3comp}
Let $u \in \cS'(\rr d)$.
\begin{enumerate}
\item If $(x_0,\xi_0) \in \rr d \times (\rr d \setminus \{0\})$ we say that
$(x_0,\xi_0) \notin WF_\psi(u)$ if there exist $\varphi_{x_0}$ and $\psi_{\xi_0}$
such that $\psi_{\xi_0}(D) ( \varphi_{x_0} u) \in \cS$;

\item If $(x_0,\xi_0) \in (\rr d \setminus \{0\}) \times \rr d$ we say that
$(x_0,\xi_0) \notin WF_e(u)$ if there exist $\psi_{x_0}$ and $\varphi_{\xi_0}$
such that $\varphi_{\xi_0}(D) ( \psi_{x_0} u) \in \cS$;

\item If $(x_0,\xi_0) \in (\rr d \setminus \{0\}) \times (\rr d \setminus \{0\})$ we say that
$(x_0,\xi_0) \notin WF_{\psi e}(u)$ if there exist $\psi_{x_0}$ and $\psi_{\xi_0}$
such that $\psi_{\xi_0}(D) ( \psi_{x_0} u) \in \cS$.
\end{enumerate}
\end{defn}

\begin{rem}
Item (1) of Definition \ref{wavefront3comp}, i.e. $WF_\psi(u)$, is the classical H\"ormander wave front set (cf. \cite[Chapter 8]{Hormander0}).
It can be shown (see \cite{Coriasco1}) that the definitions are invariant under a switch of order of the cut-off operators.
\end{rem}

The $\cS$-wave front set $WF_{\cS}(u)$ is defined as follows \cite{Coriasco1}.

\begin{defn}
Let $u \in \cS'(\rr d)$ and $(x_0,\xi_0) \in T^*(\rr d) \setminus \{ (0,0) \}$.
Then $(x_0,\xi_0) \notin WF_\cS (u)$ if one of the following conditions is satisfied:
\begin{align*}
(x_0,\xi_0) \in (\rr d \setminus \{ 0 \}) \times (\rr d \setminus \{ 0 \}) \quad & \mbox{and}
\quad (x_0,\xi_0) \notin WF_\psi (u) \cup WF_e (u) \cup WF_{\psi e} (u), \\
(x_0,\xi_0) \in \{ 0 \} \times (\rr d \setminus \{0\}) \quad & \mbox{and}
\quad (x_0,\xi_0) \notin WF_\psi (u), \quad \mbox{or} \\
(x_0,\xi_0) \in (\rr d \setminus \{0\}) \times \{ 0 \} \quad & \mbox{and}
\quad (x_0,\xi_0) \notin WF_e (u).
\end{align*}
\end{defn}

According to \cite[Theorem 3.8]{Coriasco1} we have $WF_{\cS}(u) = \emptyset$ if and only if $u \in \cS$,
similarly to the corresponding result for $WF(u)$ (cf. Corollary \ref{emptyWFchar} and \cite[Proposition 2.4]{Hormander1}).

\begin{example}\label{exdelta}
Let $u=\delta_{x_0} \in \cS'(\rr d)$ where $x_0 \in \rr d$. Then by \cite[Example 2.5]{Coriasco1} we have
$WF_e (u) = WF_{\psi e} (u) = \emptyset$ and
\begin{equation}\label{wavefrontS1}
WF_{\cS}(\delta_{x_0}) = WF_{\psi}(\delta_{x_0}) = \{ x_0 \} \times (\rr d \setminus \{ 0 \}).
\end{equation}

In order to determine $WF_G( \delta_{x_0} )$, we may assume that $x_0=0$ by Corollary \ref{shiftinvariant}.
The STFT of $\delta_{0}$ is
$$
V_\varphi \delta_{0}(x,\xi) = (\delta_{0}, M_\xi T_x \varphi ) = \overline{\varphi(-x)},
$$
where $\varphi \in \cS(\rr d) \setminus \{0\}$.
Hence $|V_\varphi \delta_{0}(0,t \xi)| = |\varphi(0)|$ for any $\xi \in \rr d$ and $t>0$,
and if we assume $\varphi(0) \neq 0$ it follows that
$$
\{ 0 \} \times (\rr d \setminus \{0\}) \subseteq WF_G(\delta_{0}).
$$
Let on the other hand $(x_0,\xi_0) \in \rr {2d} \setminus \{ (0,0) \}$ satisfy $x_0 \neq 0$.
Then $(x_0,\xi_0) \in \Gamma$ where the conic open set $\Gamma \subseteq \rr {2d} \setminus \{ (0,0) \}$ is defined by
$$
\Gamma = \{ (x,\xi) \in \rr {2d} \setminus \{ (0,0) \}: \, C|x| > |\xi| \}
$$
for some $C>0$.
We have for $(x,\xi) \in \Gamma$ and $N \geq 0$
\begin{align*}
\sup_{(x,\xi) \in \Gamma} \eabs{(x,\xi) }^N \left| V_\varphi \delta_{0} (x,\xi) \right|
\lesssim \sup_{x \in \rr d} \eabs{x}^N | \varphi(-x)| < \infty,
\end{align*}
and therefore $(x_0,\xi_0) \notin WF_G(\delta_{0})$. We conclude that
\begin{equation}\label{wavefrontglobal1}
WF_G(\delta_{x_0}) = WF_G(\delta_{0}) = \{ 0 \} \times (\rr d \setminus \{0\}).
\end{equation}
Comparing \eqref{wavefrontglobal1} with \eqref{wavefrontS1} we may draw the following two conclusions.
There are no general inclusion relations between $WF_G(u)$ and $WF_{\cS}(u)$ for $u \in \cS'$,
and in this example the classical wave front set and $WF_{\cS}(u)$ give finer information than $WF_G(u)$.
\end{example}

\begin{example}\label{exexpo}
Here we consider $u(x) = e^{i \la x,\xi_0 \ra} \in \cS'(\rr d)$ for $\xi_0 \in \rr d$.
The classical wave front set $WF_\psi(u)=\emptyset$ since $u \in C^\infty(\rr d)$, and $WF_{\psi e}=\emptyset$ (cf. \cite[Example 2.5]{Coriasco1}).
From \cite[Lemma 2.4]{Coriasco1} we have
\begin{equation*}
WF_{\cS}(u) = WF_{e}(u) = (\rr d \setminus \{ 0 \}) \times \{ \xi_0 \} .
\end{equation*}
Let us compute $WF_G(u)$. By Corollary \ref{shiftinvariant} we may assume $\xi_0=0$, that is $u \equiv 1$.
The STFT is given by $V_\fy 1(x,\xi)=(1,M_\xi T_x \fy)$ so that $|V_\fy 1(x,\xi)| = |\wh \fy(-\xi)|$.
Hence $|V_\fy 1(tx,0)| = |\wh \fy(0)|$ for any $x \in \rr d$ and $t>0$.
If $\wh \fy(0) \neq 0$ it follows that
\begin{equation*}
(\rr d \setminus \{0\}) \times  \{ 0 \}  \subseteq WF_G(1).
\end{equation*}
Arguing as in Example \ref{exdelta} we conclude
\begin{equation*}
WF_G(e^{i \la \cdot,\xi_0 \ra}) = WF_G(1) = (\rr d \setminus \{0\}) \times  \{ 0 \}.
\end{equation*}
So in this example the classical wave front set does not identify any singularity, whereas $WF_\cS(u)$ gives finer information than $WF_G(u)$.
\end{example}

\begin{example}
Let $d=1$, $c \in \ro \setminus \{ 0 \}$ and $u(x)=e^{i c x^2/2}$, $x \in \ro$.
According to \cite[Example 2.7]{Coriasco1} we have $WF_\psi(u)=WF_e(u) = \emptyset$ and
\begin{equation}\label{wavefrontS2}
WF_{\cS}(u) = WF_{\psi e}(u) = (\ro \setminus \{0\}) \times (\ro \setminus \{0\}).
\end{equation}
To compute $WF_G(u)$, consider the STFT given by $V_\fy u(x,\xi) = (u,M_\xi T_x \fy)$ where $\fy(y)=e^{-y^2/2}$.
Standard computations give
\begin{equation*}
|V_\fy u(x,\xi)| = C \exp\left( - \frac{\left( \xi-c x \right)^2}{2(1+c^2)}  \right)
\end{equation*}
with $C>0$.
Similar arguments as above show then
\begin{equation}\label{wavefrontglobal2}
WF_G(u) = \{ (x, cx): \, x \in \ro \setminus \{ 0 \} \}.
\end{equation}
Comparing \eqref{wavefrontglobal2} with \eqref{wavefrontS2} we conclude
that $WF_G(u)$ gives finer information than $WF_{\cS}(u)$ for this example.
\end{example}



\begin{thebibliography}{2000}

\bibitem{Asada1}
K.~Asada and D.~Fujiwara,
\textit{On some oscillatory transformations in $L^2(\rr n)$},
Japan J. Math. \textbf{4}, 299--361, 1978.

\bibitem{Benyi1}
\'A.~B\'enyi and K.~A.~Okoudjou,
\textit{Local well-posedness of nonlinear dispersive equations on modulation spaces},
Bull. London Math. Soc. \textbf{41} (3), 549--558, 2009.

\bibitem{Benyi2}
\'A.~B\'enyi, K.~Gr\"ochenig, K.~A.~Okoudjou and L.~G.~Rogers,
\textit{Unimodular Fourier multipliers for modulation spaces},
J. Funct. Anal. \textbf{246} (2), 366--384, 2007.

\bibitem{Benyi3}
\'A.~B\'enyi, L.~Grafakos, K.~Gr\"ochenig and K.~A.~Okoudjou,
\textit{A class of Fourier multipliers for modulation spaces},
Appl. Comput. Harmon. Anal. \textbf{19} (1), 131--139, 2005.

\bibitem{Bony1}
J.-M.~Bony,
\textit{Op\'erateurs int\'egraux de Fourier et calcul de Weyl--H\"ormander (cas d'une m\'etrique symplectique)},
Journ\'ees ``\'Equations aux D\'eriv\'ees Partielles'' (Saint-Jean-de-Monts, 1994), Exp. No. IX, pp. 1--14, \'Ecole Polytech., Palaiseau, 1994.

\bibitem{Luef1}
E.~Cordero, H.~Feichtinger and F.~Luef,
\textit{Banach Gelfand triples for Gabor analysis},
Pseudo-Differential Operators, Lecture Notes in Math. \textbf{1949},
Eds. L. Rodino, M. W. Wong,
pp. 1--33, Springer, Berlin, 2008.

\bibitem{Cordero1}
E.~Cordero and K.~Gr\" ochenig,
\textit{Time-frequency analysis of
localization operators},
J. Funct. Anal.  \textbf{205} (1), 107--131, 2003.

\bibitem{Cordero2}
E.~Cordero, F.~Nicola and L.~Rodino,
\textit{Time-frequency analysis of Fourier integral operators},
Comm. Pure Appl. Anal. \textbf{9} (1), 1--21, 2010.

\bibitem{Cordero3}
\bysame,
\textit{Sparsity of Gabor representation for Schr\"odinger propagators},
Appl. Comput. Harmon. Anal. \textbf{26} (3), 357--370, 2009.

\bibitem{Cordero4}
E.~Cordero, K.~Gr\"ochenig, F.~Nicola and L.~Rodino,
\textit{Wiener algebras of Fourier integral operators},
J. Math. Pures Appl., in press, 2012.

\bibitem{Coriasco1}
S.~Coriasco and L.~Maniccia, \textit{Wave front set at infinity and hyperbolic
linear operators with multiple characteristics},
Ann. Glob. Anal. Geom. \textbf{24}, 375–-400, 2003.

\bibitem{Daubechies1}
I.~Daubechies,
\textit{The wavelet transform, time-frequency localization and signal analysis},
IEEE Trans. Inform. Theory \textbf{36} (5), 961--1005, 1990.

\bibitem{Feichtinger1}
H.~G.~Feichtinger, \textit{Modulation spaces on locally compact abelian groups},
Technical Report, University Vienna, 1983; also in:
Wavelets and Their Applications, Eds. M. Krishna, R. Radha, S. Thangavelu, pp. 99--140, Allied Publishers, 2003.

\bibitem{FeiGro1}
H.~G.~Feichtinger and K.~Gr\"ochenig,
\textit{Gabor frames and
time-frequency analysis of distributions},
J. Funct. Anal.
\textbf{146} (2), 464--495, 1997.

\bibitem{Folland1}
G.~B.~Folland, \textit{Harmonic Analysis in Phase Space}, Princeton University Press, 1989.

\bibitem{Grochenig1}
K.~Gr\" ochenig, \textit{Foundations of Time-Frequency Analysis}, Birkh\" auser, Boston, 2001.

\bibitem{Grochenig2}
\bysame, \textit{Time-frequency analysis of Sj\"ostrand's class}, Revista Mat. Iberoam. \textbf{22} (2), 703--724, 2006.

\bibitem{Grochenig3}
K.~Gr\"ochenig and M. Leinert, \textit{Wiener's lemma for twisted
convolution and Gabor frames}, J. Amer. Math. Soc. \textbf{17} (1), 1--18, 2004.

\bibitem{Holst1}
A.~Holst, J.~Toft and P.~Wahlberg, \textit{Weyl product algebras and modulation spaces},
J. Funct. Anal. \textbf{251}, 463--491, 2007.

\bibitem{Hormander2}
L.~H\"ormander,
\textit{Fourier integral operators. I},
Acta Math. \textbf{127} (1), 79--183, 1971.

\bibitem{Hormander0}
\bysame, \textit{The Analysis of Linear Partial Differential Operators}, vol I, III,
Springer-Verlag, Berlin Heidelberg NewYork Tokyo, 1983.

\bibitem{Hormander1}
\bysame, \textit{Quadratic hyperbolic operators}, Microlocal Analysis and Applications, Lecture Notes in Math. \textbf{1495}, Eds. L. Cattabriga, L. Rodino, pp. 118--160, Springer, 1991.

\bibitem{Janssen1}
A.~J.~E.~M.~Janssen, \textit{Duality and biorthogonality for Weyl--Heisenberg frames}, J. Fourier Anal. Appl. \textbf{1} (4), 403--436, 1995.

\bibitem{Johansson1}
K.~Johansson, S.~Pilipovi\'c, N.~Teofanov and J.~Toft,
\textit{Gabor pairs, and a discrete approach to wave-front sets},
Monatsh. Math. \textbf{166} (2), 181--199, 2012.

\bibitem{Nakamura1}
S.~Nakamura,
\textit{Propagation of the homogeneous wave front set for Schr\"odinger equations},
Duke Math. J. \textbf{126} (2), 349--367, 2005.

\bibitem{Nicola1}
F.~Nicola and L.~Rodino, \textit{Global Pseudo-Differential Calculus on Euclidean Spaces}, Birkh\"auser, 2010.

\bibitem{Pilipovic1}
S.~Pilipovi\'c, N.~Teofanov and J.~Toft,
\textit{Micro-local analysis with Fourier Lebesgue spaces. Part I},
J. Fourier Anal. Appl. \textbf{17} (3), 374--407, 2011.

\bibitem{Pilipovic2}
\bysame,
\textit{Micro-local analysis with Fourier Lebesgue spaces and modulation spaces: part II},
J. Pseudo-Differ. Oper. Appl. \textbf{1} (3), 341--376, 2010.

\bibitem{Reed1}
M.~Reed and B.~Simon, \textit{Methods of Modern Mathematical Physics}, vol. 1, Academic Press, 1980.

\bibitem{Rochberg1}
R.~Rochberg and K.~Tachizawa,
\textit{Pseudodifferential operators, Gabor frames, and local trigonometric bases},
Gabor analysis and algorithms, Theory and applications, Eds. H.~G.~Feichtinger and T.~Strohmer, pp.~171-–192,
Appl. Numer. Harmon. Anal., Birkh\"auser, Boston, MA, 1998.

\bibitem{Sjostrand1}
J.~Sj\"ostrand, \textit{An algebra of pseudodifferential operators},
Math. Res. L. \textbf{1}, 185--192, 1994.

\bibitem{Sjostrand2}
\bysame, \textit{Wiener type algebras of pseudodifferential operators},
S\'eminaire Equations aux D\'eriv\'ees Partielles, Ecole Polytechnique, 1994/1995,
Expos\'e n$^{\circ}$ IV.

\bibitem{Shubin1}
M.~A.~Shubin, \textit{Pseudodifferential Operators and Spectral Theory}, Springer, 2001.

\bibitem{Strohmer1}
T.~Strohmer,
\textit{Pseudodifferential operators and Banach algebras in mobile communications},
Appl. Comput. Harmon. Anal. \textbf{20} (2), 237--249, 2006.

\bibitem{Tataru1}
D.~Tataru,
\textit{Phase space transforms and microlocal analysis},
Phase Space Analysis of Partial Differential Equations, Vol. II, pp. 505–-524, Pubbl. Cent. Ric. Mat. Ennio Giorgi, Scuola Norm. Sup., Pisa, 2004.

\bibitem{Walnut1}
D.~F.~Walnut,
\textit{Lattice size estimates for Gabor decompositions},
Monatsh. Math. \textbf{115} (3), 245--256, 1993.

\bibitem{Wang1}
B.~Wang, Z.~Lifeng and G.~Boling,
\textit{Isometric decomposition operators, function spaces $E_{p,q}^\lambda$ and applications to nonlinear evolution equations},
J. Funct. Anal. \textbf{233} (1), 1--39, 2006.

\bibitem{Wang2}
B.~Wang and C.~Huang,
\textit{Frequency-uniform decomposition method for the generalized BO, KdV and NLS equations},
J. Differential Equations \textbf{239} (1), 213--250, 2007.

\end{thebibliography}
\end{document}